\documentclass[11pt,a4paper,reqno]{amsart}
\usepackage{
    mathtools,
    amssymb,
    amsthm,
    thmtools,
    amsrefs,
    xcolor,
    xfrac,
    calc,
    xypic,
    indentfirst,
    hyperref,
    cleveref,
    aligned-overset,
    easyReview
}
\allowdisplaybreaks

\usepackage{a4wide}
\usepackage[shortlabels]{enumitem}

\newtheorem{teo}{Theorem}[section]
    \newtheorem{prop}[teo]{Proposition}
    \newtheorem{lema}[teo]{Lemma}
    
    \newtheorem{obs}[teo]{Remark}
    
    \newtheorem{ex}[teo]{Example}

\theoremstyle{definition}
    \newtheorem{dfn}[teo]{Definition}
    \newtheorem{defin}[teo]{Definition}

\newcommand{\El}{\varepsilon_t}
\newcommand{\Er}{\varepsilon_s}

\newcommand{\Ve}{\varepsilon}
\newcommand{\Sp}{\mathbin{\underline{\#}}}

\newcommand{\Hp}{{H_{_{par}}^w}}

\newcommand{\G}{\mathcal{G}}

\def \OS #1{\overset{#1}}
\renewcommand{\Ref}[1]{\overset{\textrm{ \tiny (\ref{#1})}}{=}}
    \def \OSeq #1{\OS{\eqref{#1}}}

\newcommand{\Ape}{\cdot}

\newcommand{\Aie}{\triangleright}
\newcommand{\Aid}{\triangleleft}

\newcommand{\Si}{S^{-1}}

\let\longto\longrightarrow

\newcommand{\E}{{E}}
\newcommand{\Et}{{\tilde{E}}}
\newcommand{\A}{{A_{_{par}}^w}}

\newcommand{\Ha}{\mathcal{H}}
\newcommand{\Sa}{\mathcal{S}}
\newcommand{\At}{{\tilde{A}_{_{par}}^w}}
\DeclareMathOperator{\End}{End}
\def \Cat #1{\ensuremath{\mathbf{#1}}}
\def \Catpm {\ensuremath{{}_{H}\Cat{parMod}}}
\def \CatHparm {\ensuremath{{}_{\Hp}\!\!\!\Cat{Mod}}}

\newcommand{\Not}[2][0]{	
	\setcounter{enumi}{#1}
	\renewcommand{\theenumi}{#2\arabic{enumi}}
	\renewcommand{\labelenumi}{(\theenumi)}
	\setlength{\itemindent}{\widthof{#2}}
	\setlength{\itemsep}{4pt}
}
\makeatletter
\def\Label#1#2{\begingroup
	\def\@currentlabel{#2}%
	\label{#1}\endgroup
}

\newenvironment{igualdades*}{
	\[
	\setlength\arraycolsep{1.4pt} 
	\begin{align*}{rclr}
}{
\end{align*}
\]
}

\title{Weak Partial Representations}
\author[Castro]{Felipe Castro}
\address[Felipe Castro]{Departamento de Matem\'{a}tica, Universidade Federal de Santa Catarina, Brazil.}
\email{f.castro@ufsc.br}

\author[Quadros]{Glauber Quadros}
\address[Glauber Quadros]{Coordenação Acadêmica, Universidade Federal de Santa Maria, Brazil}
\email{glauber.quadros@ufsm.br}

\author[Tamusiunas]{Thaísa Tamusiunas}
\address[Thaísa Tamusiunas]{Departamento de Matem\'{a}tica Pura e Aplicada, Universidade Federal do Rio Grande do Sul, Brazil - Corresponding author}
\email{thaisa.tamusiunas@gmail.com}

\date{}

\begin{document}\frenchspacing
\renewcommand*{\sectionautorefname}{Section}
\renewcommand*{\chapterautorefname}{Chapter}


\begin{abstract}
    We introduce the notion of partial representation of a weak Hopf algebra. We present the universal algebra $\Hp$, which factorizes these partial representations by algebra morphisms. Also, it is shown that $\Hp$ is isomorphic to a partial smash product, that it has the structure of a Hopf algebroid and also that it can be endowed with a quantum inverse semigroup structure. Moreover, it is shown that the algebra objects in the module category over \(\Hp\) correspond to symmetrical partial module algebras.
\end{abstract}

\maketitle

\

\noindent \textbf{2020 AMS Subject Classification:} Primary 16T99. Secondary 16T05, 16S40.

\noindent \textbf{Keywords:} weak Hopf algebra, partial representation, weak partial representation.

\section{Introduction}

Dokuchaev, Exel and Piccione  analyzed the behavior of partial representations of groups in \cite{DEP}. Among others results, they established a correspondence between partial representations of the group $G$ and representations of the groupoid algebra $K_{par}G$, which is defined over a groupoid $\Gamma(G)$ constructed from the group $G$. This can be viewed as a partial-to-global result, as it describes a partial representation in terms of a global representation. An important aspect about the algebra $K_{par}G$ is that it is generated by an indutive groupoid, which is a type of ordered groupoid. However, this specific groupoid has a significant characteristic: it is associated with an inverse semigroup via the classical Ehresmann-Schein-Nambooripad Theorem (ESN Theorem) \cite[Theorem 4.1.8]{lawson1998inverse}. Furthermore, by the definition of the functors in the ESN Theorem, the image
$\Gamma(G)$ is more than just an inverse semigroup; it is an inverse monoid.

These results naturally raise the question: does the same apply to structures that extend groups? One possible way to extend groups is through groupoids, where not all elements are composable with each other and there are multiple identity elements instead of a single one. Indeed, a groupoid is a small category in which all morphisms are invertible. A theory for partial representations of groupoids was developed in \cite{lata} and \cite{lata2}, showing that each partial representation of a groupoid corresponds to a unique global representation of a groupoid algebra. The algebra associated with this correspondence is generated by the Birget-Rhodes expansion of the groupoid \cite[Theorem 3.3]{lata2}, which is a type of ordered groupoid \cite[Proposition 3.1]{gilbert2005actions}. More specifically, it is a locally complete inductive groupoid \cite[Remark 2.10]{lata}, which has the additional characteristic of being associated with an inverse category, via a generalization of the ESN Theorem \cite[Theorem 2.17]{dewolf2018ehresmann}. When the groupoid is, in particular, a group, this inverse category coincides with the inverse monoid that appears in the group case.

Another possible generalization of groups is via a Hopf algebra, which can be interpreted as a generalization of groups by linearity. In \cite{ABV}, Alves, Batista and Vercruysse studied partial representations for Hopf algebras. Given a Hopf algebra $H$, they introduced a universal algebra $H_{par}$, which has the universal property of factorizing a partial representation of $H$ by an algebra morphism. Also, among other properties, they showed its intrinsic relation with a partial action of $H$. The algebra $H_{par}$ presented in \cite{ABV} was shown to be a Hopf algebroid, and it establishes a one-to-one correspondence between partial representations of Hopf algebras and representations of the constructed Hopf algebroid. This result also generalizes the correspondence for groups presented in \cite{DEP}.

The reason for citing the ESN Theorem is that in \cite{qua}, a significant step was made in the direction to understand how this theorem could work in a Hopf context. In this reference, it was introduced the concept of quantum inverse semigroups, which is an algebraic structure that generalizes inverse semigroups by linearity. Examples of quantum inverse semigroups include Hopf algebras, weak Hopf algebras, inverse semigroup algebras and inverse semigroupoid algebras (in particular, inverse monoid algebras and inverse category algebras). Also, it was proved that the constructed Hopf algebroid $H_{par}$ of \cite{ABV} admits a quantum inverse semigroup, provided that the Hopf algebra is cocommutative.

There is yet another generalization of groups, which is the focus of our work: weak Hopf algebras. Partial representations of weak Hopf algebras unify the theories of partial representations for both groupoids and Hopf algebras, as each weak Hopf algebra $H$ can be viewed as a generalization of a groupoid (by linearity) as well as of a Hopf algebra. Regarding weak Hopf algebras, a similar construction for partial representations is also possible. It was already expected, and to the theory be consistent, our initial hypothesis was that it should exists a universal algebra $H_{par}^w$ with the universal property of factorizing a partial representation of $H$ by an algebra morphism and this algebra should also possess a Hopf algebroid structure in the weak case. This hypothesis holds true, as we will demonstrate throughout the paper. Naturally, another question arises: is it possible to endow the universal algebra $\Hp$ with a quantum inverse semigroup structure? The answer is yes, provided that the weak Hopf algebra is cocommutative, and this result is shown in \autoref{QISG}. 

In summary, in a Hopf context, partial representations of Hopf algebras are equivalent to representations of a constructed Hopf algebroid \cite[Theorem 4.2 and Theorem 4.10]{ABV}; and in the weak case, partial representations of weak Hopf algebras are equivalent to representations of another constructed Hopf algebroid. Both types of Hopf algebroids (standard and weak), under the assumption of cocommutativity of $H$,  admit a quantum inverse semigroup structure. So far, the theory remains well-behaved.

Therefore, our main purpose in this paper is to introduce the concept of partial representation of a weak Hopf algebra and some of its consequences. This notion unifies the partial representations of Hopf algebras discussed in \cite{ABV} and \cite{qua}, as well as the partial representations of groupoids presented in \cite{lata} and \cite{lata2}. The paper is organized as follows. In Section 2, we introduce some preliminary results and terminology. In Section 3, we define partial representations of a weak Hopf algebra in an algebra, provide examples, and discuss their properties.  In Section 4, we present the universal algebra $\Hp$, which factorizes any partial representation of a weak Hopf algebra by an unique algebra morphism. We also show that $\Hp$ is isomorphic to a partial smash product. In Section 5, we prove that $\Hp$ has a structure of a Hopf algebroid. In Section 6, we define partial modules over weak Hopf algebras, demonstrate that the category of partial modules is isomorphic to the category of partial representations, and show that a symmetric partial module algebra corresponds to an algebra object in the module category over the associated universal Hopf algebroid. Finally, in Section 7, we prove that $\Hp$ admits a quantum inverse semigroup structure.

For readers who are less familiar with weak Hopf algebras, we refer them to \cite{WHA1} and \cite{WHA2} for additional information.

Throughout, unless otherwise stated, all algebras are modules over a field $\Bbbk$ and unadorned $\otimes$ means $\otimes_{\Bbbk}$. Rings and algebras are associative and unital. For all the paper we will adopt the Sweedler notation (summation understood).

\section{Preliminaries}

We begin by presenting the concept of a weak Hopf algebra and discussing some of its properties. The reference for this section is \cite{CPQS}.

A sixtuple $(H, m, u, \Delta, \Ve, S)$ is a weak Hopf algebra with antipode $S$ if:
	\begin{enumerate}[(i)]
		\item $(H, m, u)$ is a $\Bbbk$-algebra, 
		\item $(H, \Delta, \Ve)$ is a $\Bbbk$-coalgebra,
		\item $\Delta(kh) = \Delta(k)\Delta(h)$
		\item $\Ve(kh_1)\Ve(h_2g) = \Ve(khg) = \Ve(kh_2)\Ve(h_1g)$,
		\item $(1_H \otimes \Delta(1_H))(\Delta(1_H) \otimes 1_H) = \Delta^2(1_H) = (\Delta(1_H)\otimes 1_H)(1_H \otimes \Delta(1_H))$,
		\item $h_1S(h_2) = \Ve_t(h)$,
		\item $S(h_1)h_2 = \Ve_s(h)$,
		\item $S(h) = S(h_1)h_2S(h_3)$,
	\end{enumerate}
for all $g, h, k \in H$, where $\Ve_t\colon H \longto H$ and $\Ve_s\colon H \longto H$ are defined by $\Ve_t(h) = \Ve(1_1h)1_2$ and $\Ve_s(h) = 1_1\Ve(h1_2)$, respectively. We will use the notation $H_t = \Ve_t(H)$ and $H_s = \Ve_s(H)$.

\begin{lema}\label{p9}\label{p11}\label{p12}\label{p4}
	Let $H$ be a weak Hopf algebra. Then,
	\begin{align}
		h_1 \otimes h_2 S(h_3) &= 1_1 h \otimes 1_2 \label{p4-1}\\
		S(h_1)h_2 \otimes h_3 &= 1_1 \otimes h 1_2 \label{p4-2}\\
		S(1_H) &= 1_H \label{p11-4}\\
		h_1 \otimes S(h_2) \ h_3 &= h \ 1_1 \otimes S(1_2) \label{p12-1} \\
		h_1 \ S(h_2) \otimes h_3 &= S(1_1) \otimes 1_2 \ h \label{p12-2},
	\end{align}
 for all $h, k\in H$. 
\end{lema}

\begin{lema}\label{p3}\cite[Lemma 2.4]{CPQS}
    The following statements hold:
    \begin{equation}
        z \in H_t \iff \Delta(z) = 1_1 z \otimes 1_2 \label{p3-1}
    \end{equation}
    and in this case $\Delta(z) = z 1_1 \otimes 1_2$.
    \begin{equation}
        w \in H_s \iff \Delta(w) = 1_1 \otimes w 1_2 \label{p3-2}
    \end{equation}
    and in this case $\Delta(w) = 1_1  \otimes 1_2 w$.
    
    Moreover, $\Delta(H_t) \subset H \otimes H_t$ and $\Delta(H_s) \subset H_s \otimes H$.
\end{lema}

We recall the definition of  partial $H$-module algebra. A left partial action of the weak Hopf algebra $H$ on the algebra $A$ is a linear mapping $\alpha \colon H \otimes A \longto A$, denoted here by $\alpha(h \otimes a) = h \cdot a$, such that
\begin{enumerate}\Not{PA}
    \item $h \cdot (ab) = (h_{1} \cdot a)(h_2 \cdot b)$,
    \item $ 1_H \cdot a = a$,
    \item $h \cdot (k \cdot a) = (h_1 \cdot 1_A)((h_2k) \cdot a)$,
\end{enumerate}
for all $h, k \in H, a \in A$. In this case, we call A a \emph{left partial H-module algebra}. Moreover, we say that a partial action of $H$ on $A$ is \emph{symmetric} (or $A$ is a \emph{left symmetric partial H-module algebra}) if it has the additional property:
\begin{enumerate}\Not[3]{PA}
    \item $h \cdot (k \cdot a) = ((h_1k) \cdot a)(h_2 \cdot 1_A),$ for all $h, k \in H, a \in A$.\label{pma-sym}
\end{enumerate}

If $A$ is a partial $H$-module algebra, it follows by \cite{CPQS} that:

\begin{prop}\label{lemaadicional}	
 $(h \cdot a)(k \cdot b) = (1_1h \cdot a)(1_2 k \cdot b)$, for all $h,k \in H$ and $a,b \in A$. 
\end{prop}

\begin{prop}\label{noHRcola}
    If $w\in H_s$ (or, $w\in H_t$ and the partial action is symmetric),
    then
    \[
        w\cdot (h\cdot a)=wh\cdot a,
    \]
    for all $h\in H$ and $a\in A$. 
\end{prop}

Recall that the \emph{smash product} $A \# H$ is the tensor product $A \otimes_{H_t} H$ with multiplication rule given by $(a\#h)(b\#g) = a(h_1 \cdot b)\#h_2g$, and the \emph{partial smash product} $A \Sp H$ is the subalgebra generated by the elements of the form $(a\#h)(1_A \# 1_H) = a(h_1 \cdot 1_A)\#h_2$. By \(a \Sp h\) we mean \((a \mathbin{\#}h)(1_A \mathbin{\#} 1_H)\).

\begin{prop}
	Let $A$ be a partial $H$-module algebra. The following sentences are valid.
\begin{enumerate}	
\item 	$a \Sp h = (a \Sp 1_H)(1_A\Sp h)$, for all $a \in A$ and $h \in H$. \label{quebradosmash}
\item The following map is an algebra monomorphism. \label{imersaonosmash}
    \begin{align*}
    	\imath\colon A & \longto A\Sp H \\
    	a & \longmapsto a \Sp 1_H
	\end{align*}

	\end{enumerate}
\end{prop}

\section{Partial Representation Theory}

We start this section by introducing the notion of partial representation of a weak Hopf algebra in an algebra.

\begin{dfn}\label{defprep}
	A partial representation of a weak Hopf algebra $H$ in an algebra $A$, or simply a weak partial representation, is a $\Bbbk$-linear map $\pi: H \to A$ such that the following statements hold:
	\begin{enumerate}\Not{PR}
		\item $\pi(1_H) = 1_A$\label{defprep-1};
		\item $\pi(h) \pi(k_1) \pi(S(k_2)) = \pi(h k_1) \pi(S(k_2))$\label{defprep-2};
		\item $\pi(h) \pi(S(k_1)) \pi(k_2) = \pi(h S(k_1)) \pi(k_2)$\label{defprep-3};
		\item $\pi(h_1) \pi(S(h_2)) \pi(k) = \pi(h_1) \pi(S(h_2) k)$\label{defprep-4};
		\item $\pi(S(h_1)) \pi(h_2) \pi(k) = \pi(S(h_1)) \pi(h_2 k)$\label{defprep-5};
		\item $\pi(h) = \pi(h_1) \pi(S(h_2)) \pi(h_3)$\label{defprep-6}.
	\end{enumerate}
\end{dfn}

Observe that the definition of partial representation of Hopf algebras does not require the condition \eqref{defprep-6}, because it is a consequence from the others, as it can be seen in \cite{ABV}*{Proposition 3.3}. In the weak case we have following equivalences.

\begin{prop} \label{6-equiv}
	Let $\pi: H \to A$ a linear map and suppose that it satisfies the conditions (\ref{defprep-1}-\ref{defprep-5}). Then the following assertions are equivalent:
	\begin{enumerate}\Not{}
		\item $\pi(h) = \pi(h_1) \pi(S(h_2)) \pi(h_3)$\label{6-equiv-1};
		\item $\pi(1_1) \pi(S(1_2)) = 1_A$\label{6-equiv-2};
		\item $\pi(S(1_1)) \pi(1_2) = 1_A$\label{6-equiv-3};
		\item $\pi(S(h)) = \pi(S(h_1)) \pi(h_2) \pi(S(h_3))$\label{6-equiv-4}.
	\end{enumerate}
	
	Moreover, if $H$ has invertible antipode, then we have 3 more conditions to add to the above ones:
	\begin{enumerate}\Not[4]{}
		\item $\pi(h) = \pi(h_3) \pi(\Si (h_2)) \pi(h_1)$\label{6-equiv-1'};
		\item $\pi(\Si(1_2)) \pi(1_1) = 1_A$\label{6-equiv-2'};
		\item $\pi(1_2) \pi(\Si(1_1)) = 1_A$\label{6-equiv-3'}.
	\end{enumerate}
\end{prop}

\begin{proof}
	Since $\pi$ satisfies (\ref{defprep-1}-\ref{defprep-5}),
	
	\begin{align*}
	\pi(h_1) \pi(S(h_2)) \pi(h_3) &= \pi(h_1) \pi(S(h_2) h_3) \\
	\OSeq{p12-1}&= \pi(h 1_1) \pi(S(1_2)) \\
	\OSeq{defprep-2}&= \pi(h) \pi(1_1) \pi(S(1_2))\\
        \intertext{and}
	\pi(h_1) \pi(S(h_2)) \pi(h_3) &= \pi(h_1 S(h_2)) \pi(h_3) \\
	\OSeq{p12-2}&= \pi( S(1_1) ) \pi(1_2 h) \\
	\OSeq{defprep-5}&= \pi{(S(1_1))} \pi(1_2) \pi(h).
	\end{align*}
	
	Moreover, we also have
	\begin{align*}
	\pi(S(h_1)) \pi(h_2) \pi(S(h_3)) &= \pi(S(h_1)) \pi(h_2 S(h_3)) \\
	\OSeq{p4-1}&= \pi(S(1_1 h)) \pi(1_2) \\
	&= \pi(S(h)S(1_1)) \pi(1_2) \\
	\OSeq{defprep-3}&= \pi(S(h)) \pi(S(1_1)) \pi(1_2).
	\end{align*}
	and
	\begin{align*}
	\pi(S(h_1)) \pi(h_2) \pi(S(h_3)) &= \pi(S(h_1) h_2) \pi(S(h_3)) \\
	\OSeq{p4-2}&= \pi(1_1) \pi(S(h 1_2)) \\
	&= \pi(1_1) \pi(S(1_2) S(h)) \\
	\OSeq{defprep-4}&= \pi(1_1) \pi(S(1_2)) \pi(S(h)).
	\end{align*}
	
	Applying (\ref{defprep-1}), the result follows.
\end{proof}

Observe that if $\pi$ is a global representation (i.e., $\pi$ is an algebra morphism) then $\pi$ is also a partial representation. When $H$ is a Hopf algebra and $\pi: H \to A$ is a partial representation, we have the following equivalences (cf. \cite{ABV}):
\begin{enumerate}[(I)]
	\item $\pi(h_1) \pi(S(h_2)) = \pi(h_1 S(h_2)) = \Ve(h) 1_A, ~ \forall h \in H$ 
	\item $\pi(S(h_1)) \pi(h_2) = \pi(S(h_1) h_2) = \Ve(h) 1_A, ~ \forall h \in H$
	\item $\pi$ is a representation.
\end{enumerate}

In the weak Hopf algebra case, this theorem can be formulated in the following way.

\begin{prop}\label{prepp1}
	Let $\pi:H \to A$ be a partial representation of $H$ in an algebra $A$. Then the following statements are equivalent:
	\begin{enumerate}\Not{}
		\item $\pi(h_1) \pi(S(h_2)) = \pi(h_1 S(h_2)) =\pi(\El(h)), ~ \forall h \in H$;
            \label{prepp1-1}
		\item $\pi(S(h_1))\pi(h_2) = \pi(S(h_1) h_2)=\pi(\Er(h)), ~ \forall h \in H$;
            \label{prepp1-2}
		\item $\pi$ is a representation.
            \label{prepp1-3}
	\end{enumerate}
\end{prop}

\begin{proof}
	If $\pi$ is a representation, then (\ref{prepp1-1}) and (\ref{prepp1-2}) are trivial.
	
	Suppose that (\ref{prepp1-1}) holds. So, for all $h, k\in H$, we have that
	\begin{align*}
	\pi(h) \pi(k)
        \OSeq{defprep-6}&= \pi(h_1) \pi(S(h_2)) \pi(h_3)\pi(k)\\
	\OSeq{defprep-5}&= \pi(h_1) \pi(S(h_2)) \pi(h_3 k)\\
	\OSeq{prepp1-1}&= \pi(h_1 S(h_2)) \pi(h_3 k)\\
	\OSeq{p12-2}&= \pi(S(1_1)) \pi(1_2 hk)\\
	\OSeq{defprep-4}&= \pi(S(1_1)) \pi(1_2) \pi(hk)\\
	\OS{\ref{6-equiv}\eqref{6-equiv-3}}&= \pi(kh),
	\end{align*}
	then $\pi$ is a representation.
	
	Supposing now that (\ref{prepp1-2}) holds, then for all $h, k\in H$ we have that
	\begin{align*}
	\pi(h) \pi(k)
        \OSeq{defprep-6}&= \pi(h) \pi(k_1) \pi(S(k_2)) \pi(k_3)\\
	\OSeq{defprep-2}&= \pi(h k_1) \pi(S(k_2)) \pi(k_3)\\
	\OSeq{prepp1-2}&= \pi(h k_1) \pi(S(k_2) k_3)\\
	\OSeq{p12-1}&= \pi(h k 1_1) \pi(S(1_2))\\
	\OSeq{defprep-2}&= \pi(h k) \pi(1_1) \pi(S(1_2))\\
	\OS{\ref{6-equiv}\eqref{6-equiv-2}}&= \pi(h k),
	\end{align*}
	and it means that $\pi$ is a representation.
\end{proof}

In  \autoref{noHRcola} it has be shown that for any partial $H$-module algebra $A$, the action of elements in $H_s$ and, in some cases, elements of $H_t$ on $A$, behaves like a global action. The next proposition is an analogous statement for partial representations:

\begin{prop}\label{HRHLprep}
	Let $w \in H_s$, $z \in H_t$ and $\pi: H \to A$ a partial representation. Then for all $h \in H$:
	\begin{enumerate}\Not{}
		\item $\pi(w)\pi(h) = \pi(wh)$;\label{HRHLprep-1}
		\item $\pi(h)\pi(z) = \pi(hz)$;\label{HRHLprep-2}
		\item $\pi(z)\pi(h) = \pi(zh)$;\label{HRHLprep-3}
		\item $\pi(h)\pi(w) = \pi(hw)$.\label{HRHLprep-4}
	\end{enumerate}
\end{prop}

\begin{proof} Items (\ref{HRHLprep-1}) and (\ref{HRHLprep-2}) can be shown directly as follows:
	\begin{align*}
	\pi(w)\pi(h)
	& = \pi(w_1)\pi(S(w_2))\pi(w_3)\pi(h) \\
	& = \pi(w_1)\pi(S(w_2))\pi(w_3 h) \\
	& = \pi(1_1)\pi(S(1_2))\pi(1_3 w h) \\
	& = \pi(1_1)\pi(S(1_2))\pi(1_3) \pi(w h) \\
	& = \pi(w h)
	\end{align*}
	and
	\begin{align*}
	\pi(h)\pi(z)
	& = \pi(h)\pi(z_1)\pi(S(z_2))\pi(z_3) \\
	& = \pi(h z_1)\pi(S(z_2))\pi(z_3) \\
	& = \pi(h z 1_1)\pi(S(1_2))\pi(1_3) \\
	& = \pi(h z) \pi(1_1)\pi(S(1_2))\pi(1_3) \\
	& = \pi(h z).
	\end{align*}
	
	For the Items (\ref{HRHLprep-3}) and (\ref{HRHLprep-4}), remember that $S_R:H_s \to H_t$ and  $S_L:H_t \to H_s$ are isomorphisms, then we can suppose $w = S(z')$ and $z = S(w')$ where $w' \in H_s$ and $z' \in H_t$. Therefore, 
	\begin{align*}
	\pi(z)\pi(h)
	& = \pi(S(w'))\pi(h)\\
	& = \pi(S(w'_1))\pi(w'_2)\pi(S(w'_3))\pi(h)\\
	& = \pi(S(w'_1))\pi(w'_2)\pi(S(w'_3) h)\\
	& = \pi(S(1_1))\pi(1_2)\pi(S(w' 1_3) h)\\
	& = \pi(S(1_1))\pi(1_2)\pi(S(1_3) S(w')  h)\\
	& = \pi(S(1_1))\pi(1_2)\pi(S(1_3)) \pi(S(w')  h)\\
	& = \pi(S(w')  h)\\
	& = \pi(z h)\\
        \intertext{and}
	\pi(h)\pi(w)
	& = \pi(h)\pi(S(z'))\\
	& = \pi(h)\pi(S(z'_1))\pi(z'_2)\pi(S(z'_3))\\
	& = \pi(h S(z'_1))\pi(z'_2)\pi(S(z'_3))\\
	& = \pi(h S(1_1 z'))\pi(1_2)\pi(S(1_3))\\
	& = \pi(h S(z') S(1_1))\pi(1_2)\pi(S(1_3))\\
	& = \pi(h S(z'))\pi(S(1_1))\pi(1_2)\pi(S(1_3))\\
	& = \pi(h S(z'))\\
	& = \pi(h w),
	\end{align*}
	as desired.
\end{proof}

Now we are able to present some examples of partial representations.

\begin{prop}
	Let $A$ a symmetric partial $H$-module algebra. Then
	\begin{align*}
	\pi\colon H & \longto \End(A) \\
	h & \longmapsto \pi(h)(a) = h \Ape a
	\end{align*}
	is a partial representation of $H$ on the algebra $End(A)$.
\end{prop}

\begin{proof}
\eqref{defprep-1} Let $a \in A$, $\pi(1_H)(a) = 1_H \Ape a = a$.

\eqref{defprep-2} Let $h,k \in H$, then
\begin{align*}
[\pi(h) \pi(k_1) \pi(S(k_2))](a)
    &= h \Ape k_1 \Ape S(k_2) \Ape a \\
    &= h \Ape (k_1 \Ape 1_A)(k_2 S(k_3) \Ape a) \\
    \OSeq{p4-1}&= h \Ape (1_1 k \Ape 1_A)(1_2 \Ape a) \\
    \OSeq{lemaadicional}&= h \Ape (k \Ape 1_A)a \\
    &= (h_1 k \Ape 1_A)(h_2 \Ape a) \\
    &= (h_1 1_1 k \Ape 1_A)(h_2 1_2 \Ape a) \\
    &= (h_1 k_1 \Ape 1_A)(h_2 k_2 S(k_3) \Ape a) \\
    &= h k_1 \Ape [1_A(S(k_2) \Ape a)] \\
    &= h k_1 \Ape S(k_2) \Ape a \\
    &= [\pi(h k_1) \pi(S(k_2))] (a).
\end{align*}

\eqref{defprep-3} Let $h,k \in H$ and $a\in A$, then
\begin{align*}
[\pi(h) \pi(S(k_1)) \pi(k_2)](a)
    & = h \Ape S(k_1) \Ape k_2 \Ape a \\
    \OSeq{pma-sym} & = h \Ape [(S(k_2) k_3 \Ape a)(S(k_1) \Ape 1_A)]\\
    \OSeq{p12-1}&= h \Ape [(S(1_2) \Ape a)(S(k 1_1) \Ape 1_A)]\\
    & = h \Ape [(S(1_2) \Ape a)(S(1_1) S(k) \Ape 1_A)]\\
    & = h \Ape [S(1) \Ape (a(S(k) \Ape 1_A))]\\
    \OSeq{p11-4}&= h \Ape [1_H \Ape (a(S(k) \Ape 1_A))]\\
    & = h 1_H \Ape (a(S(k) \Ape 1_A))\\
    \OSeq{p11-4}&= h S(1) \Ape (a(S(k) \Ape 1_A))\\
    & = (h_1 S(1_2) \Ape a)(h_2 S(1_1) \Ape S(k) \Ape 1_A))\\
    & = (h_1 S(1_3) \Ape a)(h_2 S(1_2) \Ape 1_A)(h_3 S(1_1) S(k) \Ape 1_A))\\
    & = (h_1 S(1_2) \Ape a)(h_2 S(1_1) S(k) \Ape 1_A))\\
    & = (h_1 S(1_2) \Ape a)(h_2 S(k 1_1)\Ape 1_A))\\
    \OSeq{p12-1}&= (h_1 S(k_2)k_3 \Ape a)(h_2 S(k_1)\Ape 1_A))\\
    & = h S(k_1)\Ape (k_2 \Ape a)\\
    & = [\pi(h S(k_1)) \pi(k_2)](a)
\end{align*}

\eqref{defprep-4} Let $h,k \in H$ and $a\in A$, then
\begin{align*}
[\pi(h_1) \pi(S(h_2)) \pi(k)](a)
    &= h_1 \Ape S(h_2) \Ape k \Ape a \\
    &= (h_1 \Ape 1_A)(h_2 S(h_3) \Ape k \Ape a) \\
    \OSeq{p4-1}&=  (1_1 h \Ape 1_A)(1_2 \Ape k \Ape a)\\
    \OSeq{noHRcola}&= (1_1 \Ape h \Ape 1_A)(1_2 \Ape k \Ape a)\\
    &=(1_1 h \Ape 1_A)(1_2 \Ape 1_A)(1_3 \Ape k \Ape a)\\
    &=(1_1 h \Ape 1_A)(1_2 k \Ape a)\\
    \OSeq{p4-1}&= (h_1 \Ape 1_A)( h_2 S(h_3) k \Ape a) \\
    &= h_1 \Ape (S(h_2) k \Ape a) \\
    &= [\pi(h_1) \pi(S(h_2)k)](a)
\end{align*}

\eqref{defprep-5} Let $h,k \in H$ and $a \in A$:
\begin{align*}
[\pi(S(h_1)) \pi(h_2) \pi(k)](a)
    &= S(h_1) \Ape h_2 \Ape k \Ape a \\
    \OSeq{pma-sym} &= (S(h_2)h_3 \Ape k \Ape a)(S(h_1) \Ape 1_A) \\
    \OSeq{p12-1}&=  (S(1_2) \Ape k \Ape a)(S(h 1_1) \Ape 1_A)\\
    \OSeq{noHRcola}&=  (S(1_2) k \Ape a)(S(h 1_1) \Ape 1_A)\\
    \OSeq{p12-1}&= (S(h_2)h_3 k \Ape a)(S(h_1) \Ape 1_A) \\
    \OSeq{pma-sym}&= S(h_1) \Ape (h_2k \Ape a) \\
    &= [\pi(S(h_1)) \pi(h_2k)](a)
\end{align*}

\eqref{defprep-6} If $a \in A$,
\begin{align*}
\pi(1_1)\pi(S(1_2))(a)
    & = 1_1 \Ape S(1_2) \Ape a\\
    \OSeq{noHRcola} &= 1_1 S(1_2) \Ape a\\
    & = \El(1) \Ape a \\
    & = a\\
    & = Id_A(a)
\end{align*}
and using \autoref{6-equiv} we obtain the desired result.
\end{proof}

\begin{prop}
	Let $A$ a symmetric partial $H$-module algebra. Then
	\begin{align*}
	\pi_0:H & \longto A \underline{\#} H \\
	h & \longmapsto 1_A \underline{\#} h
	\end{align*}
	is a partial representation.
\end{prop}

\begin{proof}
(\ref{defprep-1}) $\pi_0(1_H) = 1_A \underline{\#} 1_H = 1_{A \underline{\#} H}$

(\ref{defprep-2}) Let $h,k \in H$,
\begin{align*}
\pi_0(h) \pi_0(k_1) \pi_0(S(k_2))
&= (1_A \Sp h)(1_A \Sp k_1)(1_A \Sp S(k_2)) \\
&= [(h_1 \Ape 1_A) \# h_2 k_1](1_A \Sp S(k_2)) \\
&= (h_1 \Ape 1_A) (h_2 k_1 \Ape 1_A) \Sp h_3 k_2 S(k_3)\\
\OSeq{p4-1}&= (h_1 \Ape 1_A) (h_2 1_1 k \Ape 1_A) \Sp h_3 1_2\\
&= (h_1 \Ape 1_A) (h_2 k \Ape 1_A) \Sp h_3\\
&= (h_1 \Ape k \ 1_A) \Sp h_2\\
&= (h_1 \Ape k \ 1_A)(h_2 \Ape 1_A) \Sp h_3\\
\OSeq{pma-sym} &= (h_1 k \Ape 1_A) \Sp h_2\\
&= (h_1 1_1 k \Ape 1_A) \Sp h_2 1_2\\
\OSeq{p4-1}&= (h_1 k_1 \Ape 1_A) \Sp h_2 k_2 S(k_3)\\
&= (1 \# h k_1)(1 \Sp S(k_2))\\
&= \pi_0(h k_1) \pi_0(S(k_2)).
\end{align*}

(\ref{defprep-3}) Let $h,k \in H$,
\begin{align*}
\pi_0(h) \pi_0(S(k_1)) \pi_0(k_2)
    &= (1_A \Sp h)(1_A \Sp S(k_1))(1_A \Sp k_2) \\
    &= [(h_1 \Ape 1_A) \# h_2S(k_1)](1_A \Sp k_2) \\
    &= (h_1 1_1 \Ape 1_A) (h_2 1_2 S(k_2) \Ape 1_A) \Sp h_3 S(k_1)k_3 \\
    &= (h_1 S(1_2) \Ape 1_A) (h_2 S(1_1) S(k_2) \Ape 1_A) \Sp h_3 S(k_1)k_3 \\
    &= (h_1 S(1_2) \Ape 1_A) (h_2 S(k_2 1_1) \Ape 1_A) \Sp h_3 S(k_1)k_3 \\
    \OSeq{p12-1} &= (h_1 S(k_3)k_4 \Ape 1_A) (h_2 S(k_2) \Ape 1_A) \Sp h_3 S(k_1)k_5 \\
    \OSeq{pma-sym} &= (h_1 S(k_2) \Ape (k_3 \Ape 1_A)) \Sp h_2 S(k_1)k_5 \\
    &= (h_1 S(k_3) \Ape 1_A) (h_2 S(k_2) k_4 \Ape 1_A)) \Sp h_3 S(k_1)k_5 \\
    &= (h_1 S(k_2) \Ape 1_A) \Sp h_2 S(k_1)k_3 \\
    &= \pi_0(h S(k_1)) \pi_0(k_2).
\end{align*}

(\ref{defprep-4}) Let $h,k \in H$,
\begin{align*}
\pi_0 (h_1) \pi_0 (S(h_2)) \pi_0 (k) &= (1_A \Sp h_1) (1_A \Sp S(h_2)) (1_A \Sp k)\\
&= (h_1 \Ape 1_A \# h_2 S(h_3)) (1_A \Sp k)\\
\OSeq{p4-1}&= (1_1 h \Ape 1_A \# 1_2) (1_A \Sp k)\\
&= (1_1 h \Ape 1_A)(1_2 \Ape 1_A)\Sp 1_3 k\\
\OSeq{pma-sym}&= 1_1 \Ape (h \Ape 1_A)\Sp 1_2 k\\
\OSeq{noHRcola}&= (1_1 h \Ape 1_A)\Sp 1_2 k\\
\OSeq{p4-1}&= (h_1 \Ape 1_A)\Sp h_2 S(h_3) k\\
&= (1_A \Sp h_1)(1_A \Sp S(h_2) k)\\
&= \pi_0 (h_1)\pi_0 (S(h_2) k)
\end{align*}

(\ref{defprep-5}) Let $h,k \in H$,
\begin{align*}
\pi_0 (S(h_1))\pi_0 (h_2) \pi_0 (k)
&= (1_A \Sp S(h_1))(1_A \Sp h_2) (1_A \Sp k)\\
&= (1_A \# S(h_1))(h_2 \Ape 1_A \Sp h_3 k)\\
&= S(h_2) \Ape (h_3 \Ape 1_A) \Sp S(h_1) h_4 k)\\
\OSeq{pma-sym}&= (S(h_3) h_4 \Ape 1_A)(S(h_2) \Ape 1_A) \Sp S(h_1) h_5 k)\\
\OSeq{p12-1}&= (S(1_2) \Ape 1_A)(S(h_2 1_1) \Ape 1_A) \Sp S(h_1) h_3 k)\\
&= (S(1_2) \Ape 1_A)(S(1_1)S(h_2) \Ape 1_A) \Sp S(h_1) h_3 k)\\
&= S(1_H) \Ape (S(h_2) \Ape 1_A) \Sp S(h_1) h_3 k)\\
\OSeq{p11-4}&= S(h_2) \Ape 1_A \Sp S(h_1) h_3 k)\\
&= (1_A \Sp S(h_1))(1_A \Sp h_2 k)\\
&= \pi_0 (S(h_1))\pi_0 (h_2 k)
\end{align*}

(\ref{defprep-6}) Note that
\begin{align*}
\pi_0(1_1)\pi_0(S(1_2))
&= (1 \Sp 1_1)(1 \Sp S(1_2)) \\
&= 1_1 \Sp 1_2 S(1_3) \\
\OSeq{p4-1}&= 1_1 \Ape 1 \Sp 1_2 \\
&= 1 \Sp 1 
\end{align*}
and using \autoref{6-equiv} we obtain the desired result.
\end{proof}

Next, we shall provide examples of partial representations of weak Hopf algebras. Firstly, we are going to construct a weak Hopf algebra that is not a standard Hopf Algebra. To do this, we will use the Sweedler's Hopf algebra framework as our basis. Our main objective is to define partial representations of this weak Hopf algebra which are not (global) representations.

\begin{ex} \label{exwha}
    Let $H_4$ be the Sweedler's Hopf algebra generated by $\Gamma = \{1,x,g,h\}$, where $x^2 = 0$, $g^2 = 1$ and $h = xg = -gx$. Define $H$ as the algebra generated by $\{e_a,f_b ~|~ a,b \in \Gamma\}$ with the product defined by $e_a e_b = e_{ab}$, $f_a f_b = f_{ab}$ and $e_a f_b = 0$.

    It is evident that $H$ forms a coalgebra with the structure defined by $\Delta(e_a) = e_{a_1} \otimes e_{a_2}$ and $\Delta(f_a) = f_{a_1} \otimes f_{a_2}$ for each $a \in \Gamma$, where $a_1 \otimes a_2$ represents the coproduct of $a$ in $H_4$. Furthermore, a counit defined by $\varepsilon(e_a) = \varepsilon(f_a) = 1$ if $a = 1$ or $a = g$, and $\varepsilon(e_a) = \varepsilon(f_a) = 0$ if $a = x$ or $a = h$, makes $H$ a weak bialgebra, which is not a bialgebra. Lastly, define $S(e_a) = e_{S(a)}$ and $S(f_a) = f_{S(a)}$ for $a \in \Gamma$, where $S(a)$ is the antipode of $a$ in $H_4$. It is clear that $H$ qualifies as a weak Hopf algebra, which is not a Hopf algebra.
\end{ex}

The example above presents a weak Hopf algebra with few terms, constructed to provide examples of partial representations. However, we highlight that it is a particular case of a more general construction for a $\Bbbk$-linear Hopf category, as seen in \cite[Proposition 7.1]{cbv}.

\begin{ex}
    Let $H$ be the weak Hopf algebra constructed in \autoref{exwha}. Let $e_1 \cdot 1_{\mathbb{K}} = f_1 \cdot 1_{\mathbb{K}} = 1$ and $e_a \cdot 1_{\mathbb{K}} = f_a \cdot 1_{\mathbb{K}} = 0$ if $a=x$, $a=g$ or $a=h$, a partial action of $H$ on its ground field. 
    It is straightforward to check that $\pi \colon H \to \mathbb{K} \Sp H$ defined by $\pi(e_a) = 1_{\mathbb{K}} \Sp e_a$ and $\pi(f_a) = 1_{\mathbb{K}} \Sp f_a$ for $a=1$ or $a=h$ and $\pi(e_a) = \pi(f_a) = 0$ if $a=g$ or $a=x$ is a partial representation of $H$. Note that $\pi$ is not a representation itself.
\end{ex}

\begin{ex}
     Consider $H$ as the weak Hopf Algebra defined in \autoref{exwha}. It is clear that $\pi \colon H \to \mathbb{K}$ is a partial representation of $H$ in its ground field $\mathbb{K}$ with $\pi(e_1) = \pi(f_1) = 1$ and $\pi(e_a) = \pi(f_a) = 0$ if $a=x$, $a=g$ or $a=h$. 
\end{ex} 

\begin{ex}
    Let $\mathbb{Q}$ be the quaternion algebra and $H$ as defined in $\autoref{exwha}$, but now we use the ground field as $\mathbb{C}$. Define $\pi \colon H \to \mathbb{Q}$ by $\pi(e_1) = \pi(f_1) = 1$, $\pi(e_g) = \pi(f_g) = 0$, and $\pi(e_x) = \pi(f_x) = i = -\pi(e_h) = -\pi(f_h)$. It is clear that $\pi$ is a partial representation of $H$ in $\mathbb{Q}$ that is not a representation.
\end{ex}

As well as in the case of partial representations of Hopf algebras (also partial representations of groups), the construction of the above examples are related by a natural transformation. To obtain this relation, we introduce the notion of covariant pair.

\begin{dfn}\label{defcp}
	Let $A$ be a partial $H$-module algebra and $B$ an algebra. A {\it covariant pair} associated to these data is a pair of maps $(\phi, \pi)$, where $\phi:A \to B$ is an algebra morphism and $\pi: H \to B$ a partial representation, such that, for all $h\in H$ and $a\in A$
	\begin{enumerate}\Not{CP}
		\item $\phi(h \Ape a) = \pi(h_1) \phi(a) \pi(S(h_2))$; \label{defcp-1}
		\item $\phi(a) \pi(S(h_1)) \pi(h_2) = \pi(S(h_1)) \pi(h_2) \phi(a)$.\label{defcp-2}
	\end{enumerate}
\end{dfn}

\begin{obs}
	If $A$ be a partial $H$-module algebra, then $(\phi_0, \pi_0)$ is a covariant pair associated, where $\phi_0:A \to A \Sp H$ is the canonical immersion.
\end{obs}

Indeed, let $h\in H$ and $a\in A$, so

(\ref{defcp-1})
\begin{align*}
\pi_0 (h_1) \phi_0 (a) \pi_0(S(h_2))
& = (1 \Sp h_1)(a \Sp 1_H)(1 \Sp S(h_2))\\
& = (h_1 \Ape a \# h_2)(1 \Sp S(h_3))\\
& = (h_1 \Ape a)(h_2 \Ape 1_A) \Sp h_3 S(h_4)\\
& = (h_1 \Ape a) \Sp h_2 S(h_3)\\
\OSeq{p4-1} &= (1_1 h \Ape a) \Sp 1_2\\
& = (1_A \# 1_H)[(h \Ape a) \Sp 1_H]\\
& = \phi_0 (h \Ape a)
\end{align*}

(\ref{defcp-2})

\begin{align*}
\pi_0 (S(h_1)) \Pi_0 (h_2) \phi_0 (a)
& = (1_A \Sp S(h_1)) (1_A \Sp h_2) (a \Sp 1_H)\\
& = (1_A \# S(h_1)) (1_A \# h_2) (a \Sp 1_H)\\
& = (S(h_2) \Ape 1_A \# S(h_1) h_3) (a \Sp 1_H)\\
& = (S(h_3) \Ape 1_A) (S(h_2) h_4 \Ape a) \Sp S(h_1) h_5\\
& = (S(h_3) \Ape 1_A) (S(h_2) h_4 \Ape a) \Sp S(h_1) h_5\\
& = [S(h_2) \Ape (h_3 \Ape a)] \Sp S(h_1) h_4\\
\OSeq{pma-sym}&= (S(h_3) h_4 \Ape a)(S(h_2)\Ape 1_A) \Sp S(h_1) h_5\\
\OSeq{p12-1}&= (S(1_2) \Ape a)(S(h_2 1_1)\Ape 1_A) \Sp S(h_1) h_3\\
& = (S(1_2) \Ape a)(S(1_1) S(h_2) \Ape 1_A) \Sp S(h_1) h_3\\
& = S(1_H) \Ape [a(S(h_2) \Ape 1_A)] \Sp S(h_1) h_3\\
& = a(S(h_2) \Ape 1_A) \Sp S(h_1) h_3\\
& = a (1_1 S(h_2) \Ape 1_A) \Sp 1_2 S(h_1) h_3\\
\OSeq{noHRcola}&= a [1_1 \Ape (S(h_2) \Ape 1_A)] \Sp 1_2 S(h_1) h_3\\
& = (a \# 1_H)[(S(h_2) \Ape 1_A) \Sp S(h_1) h_3]\\
& = (a \Sp 1_H) (1_A \Sp S(h_1)) (1_A \Sp h_2)\\
& = \phi_0 (a) \pi_0 (S(h_1)) \Pi_0 (h_2) 
\end{align*}

\begin{lema}\label{lemacp}
	Let $(\phi, \pi)$ be a covariant pair, then for all $h \in H$
	\begin{enumerate}\Not{}
		\item $\pi(h) = \pi(h_1) \pi(S(h_2)) \pi(h_3)$ \label{lcp-1};
		\item $\pi(S(h)) = \pi(S(h_1)) \pi(h_2) \pi(S(h_3))$ \label{lcp-2}.
	\end{enumerate}
\end{lema}

\begin{proof}
	Note that $\phi(a) = \phi(1_H \Ape a) = \pi(1_1) \phi(a) \pi(S(1_2))$. Taking $a = 1_A$ we obtain $1_H = \pi(1_1) \pi(S(1_2))$.
	
	Firstly we show (\ref{lcp-1}):
	\begin{align*}
	\pi(h)
        &= \pi(h) 1_H \\
	&= \pi(h) \pi(1_1) \pi(S(1_2)) \\
	&= \pi(h 1_1) \pi(S(1_2)) \\
	\OSeq{p12-1}&= \pi(h_1) \pi(S(h_2) h_3) \\
	&= \pi(h_1) \pi(S(h_2)) \pi(h_3).
	\end{align*}
        
	Now, for (\ref{lcp-2}):
	\begin{align*}
	\pi(S(h))
            &= 1_H \pi(S(h)) \\
    	&= \pi(1_1) \pi(S(1_2)) \pi(S(h)) \\
    	&= \pi(1_1) \pi(S(1_2) S(h)) \\
    	&= \pi(1_1) \pi(S(h 1_2)) \\
    	\OSeq{p4-2}&= \pi(S(h_1)h_2) \pi(S(h_3)) \\
    	&= \pi(S(h_1)) \pi(h_2) \pi(S(h_3)).\qedhere
	\end{align*}
\end{proof}

\begin{teo}
	Let A be a symmetric partial $H$-module algebra, $\pi: H \to B$ a partial representation and $\phi:A \to B$ an algebra morphism such that $(\phi, \pi)$ is a covariant pair. So there exists a unique algebra morphism $\Phi: A \Sp H \to B$ such that $\Phi \circ \pi_0 = \pi$ and $\Phi \circ \phi_0 = \phi$, i.e., the following diagram is commutative:
	\[
	\xymatrix{
		H\ar[r]^\pi \ar[d]_{\pi_0}\ar@{}[dr]|<<<<{\circlearrowright}|>>>>{\circlearrowright } & B \\
		A \Sp H\ar@{.>}[ru]|\Phi & A\ar[u]_{\phi} \ar[l]^{\phi_0}
	}
	\]
\end{teo}

\begin{proof}
	Define 
	\begin{align*}
	\tilde{\Phi}\colon  A \times H & \longto B\\
	(a,h) & \longmapsto \phi(a) \pi(h).
	\end{align*}
	
	Let us check that $\tilde{\Phi}$ is $H_t$ balanced. For all $z \in H_t$,
	\begin{align*}
	\tilde{\Phi}(a \triangleleft z,h)
        &= \phi(a \triangleleft z) \pi(h)\\
	&= \phi(a (S_R^{-1}(z) \Ape 1_A)) \pi(h)\\
	&= \phi(a) \phi((S_R^{-1}(z) \Ape 1_A)) \pi(h)\\
	\OSeq{defcp-1}&= \phi(a) \pi({S_R^{-1}(z)}_1) \phi( 1_A) \pi(S({S_R^{-1}(z)}_2)) \pi(h)\\
	&= \phi(a) \pi({S_R^{-1}(z)}_1) \pi(S({S_R^{-1}(z)}_2)) \pi(h)\\
	&= \phi(a) \pi({S_R^{-1}(z)}_1) \pi(S({S_R^{-1}(z)}_2) h) \\
	\OSeq{p3-2}&= \phi(a) \pi(1_1) \pi(S(S_R^{-1}(z) 1_2) h) \\
	&= \phi(a) \pi(1_1) \pi(S(1_2) S(S_R^{-1}(z)) h) \\
	&= \phi(a) \pi(1_1) \pi(S(1_2) z h) \\
	&= \phi(a) \pi(1_1) \pi(S(1_2)) \pi(z h) \\
	&= \phi(a) \pi(1_1) \phi(1_A) \pi(S(1_2)) \pi(z h) \\
	\OSeq{defcp-1}&= \phi(a) \phi(1_H \Ape 1_A) \pi(z h) \\
	&= \phi(a) \phi(1_A) \pi(z h) \\
	&= \phi(a) \pi(z h) \\
	&= \tilde{\Phi}(a,z h).
	\end{align*}
	
	Then the map 
	\begin{align*}
	\Phi: A \# H & \longto B\\
	a\# h & \longmapsto \phi(a) \pi(h)
	\end{align*}
	is well-defined. Moreover, it is an algebra morphism. In fact,
	\begin{align*}
	\Phi[(a \# h)(b \# g)]
        &= \Phi[a(h_1 \Ape b)\#h_2g] \\
	&= \phi(a(h_1 \Ape b)) \pi(h_2g) \\
	&=  \phi(a)\phi(h_1 \Ape b) \pi(h_2g) \\
	\OSeq{defcp-1}&= \phi(a)\pi(h_1) \phi(b) \pi(s(h_2)) \pi(h_3g) \\
	&= \phi(a)\pi(h_1) \phi(b) \pi(s(h_2)) \pi(h_3) \pi(g) \\
	\OSeq{defcp-2}&= \phi(a)\pi(h_1) \pi(S(h_2)) \pi(h_3) \phi(b) \pi(g) \\
	\OS{\ref{lemacp}\eqref{lcp-1}}&= \phi(a) \pi(h) \phi(b) \pi(g) \\
	&= \Phi(a \# h)\Phi(b \# g).
	\end{align*}
	
	Notice that $\Phi(1_A \# 1_H) = 1_B$. Thus $\Phi(a \Sp h) = \Phi(a \# h)$.
	
	Furthermore, $\Phi(\phi_0(a)) = \Phi(a \Sp 1_H) = \Phi(a \# 1) =  \phi(a)\pi(1_H) = \phi(a)$ and $\Phi(\pi_0(h)) = \Phi(1_A \Sp h) = \Phi(1_A \# h) = \phi(1_A) \pi(h) = \pi(h)$ and it shows the existence.
	
	For the uniqueness, suppose $\Psi:A \Sp H  \to  B$  such that $\Psi \circ \pi_0 = \pi$ and $\Psi \circ \phi_0 = \phi$. Then
	\begin{align*}
	\Psi(a \Sp h)
        &= \Psi[(a \Sp 1_H)(1_A \Sp h)] \\
	&= \Psi(a \Sp 1_H)\Psi(1_A \Sp h) \\
	&= \Psi(\phi_0(a))\Psi(\pi_0(h)) \\
	&= \phi(a)\pi(h)\\
	&= \Phi(a \Sp h).
	\end{align*}
	and it shows that $\Phi$ is  unique.	
\end{proof}

\section{The universal algebra \texorpdfstring{\(\Hp\)}{H\scriptsize par}}

In \cite{ABV}, it was constructed an algebra associated with a given Hopf algebra that factorizes partial representations by algebra morphisms. In the weak case, it is also possible to construct such an algebra, as we shall see in the following.

\begin{dfn}
	Let $H$ a weak Hopf algebra. Consider $T(H)$ the tensorial algebra, where $H$ is seen as a vector space. Take $I$ the ideal of $T(H)$ generated by the relations:
	\begin{enumerate}\Not{wHpar}\label{defHpar}
		\item $1_H - 1_{T(H)}$;\label{defHpar-1}
		\item $h \otimes k_1 \otimes S(k_2) - h k_1 \otimes S(k_2)$;\label{defHpar-2}
		\item $h \otimes S(k_1) \otimes k_2 - h S(k_1) \otimes k_2$;\label{defHpar-3}
		\item $h_1 \otimes S(h_2) \otimes k - h_1 \otimes S(h_2) k$;\label{defHpar-4}
		\item $S(h_1) \otimes h_2 \otimes k - S(h_1) \otimes h_2 k$;\label{defHpar-5}
		\item $h - h_1 \otimes S(h_2) \otimes h_3$.\label{defHpar-6}
	\end{enumerate}
	
	Then define $\Hp = \sfrac{T(H)}{I}$. We denote the coset of an element $h$ by  $[h]$.
\end{dfn}

Notice that $\Hp$ is an algebra and clearly we have that
\begin{align*}
[~]: H & \longto \Hp \\
h & \longmapsto [h]
\end{align*}
is a partial representation of $H$ on $\Hp$.

Next we show that each partial representation $\pi$ of $H$ in an algebra $B$ can be factorized by an unique algebra morphism $\hat\pi$ from $\Hp$ to $B$. This also shows that there is a one-to-one correspondence between the partial representations of $H$ and the representations of $\Hp$.

\begin{teo}\label{universaldeHpar}
	Let $\pi:H \to B$ a partial representation. Then there exist a unique $\hat \pi:\Hp \to B$ algebra morphism such that $\pi = \hat \pi \circ [~]$. In other words, the following diagram is commutative:
	\[
	\xymatrix{
		H\ar[r]^\pi \ar[d]_{[~]}\ar@{}[dr]|<<<<{\circlearrowright} & B \\
		\Hp\ar@{.>}[ru]|{\hat\pi} & 
	}
	\]
	
	Conversely, given an algebra morphism $\phi: \Hp \to B$, there is a unique partial representation $\pi_\phi:H \to B$ such that $\hat\pi_\phi = \phi$.
\end{teo}

\begin{proof}
	Let $\pi:H \to B$ be a partial representation. By the universal property of $T(H)$ there is a unique algebra morphism $\pi': T(H) \to B$ such that $\pi(h) = \pi'(h)$, for all $h\in H$.
	
	Observe that $\pi'(x \otimes y \otimes z) = \pi(x)\pi(y)\pi(z)$ and since $\pi$ is a  partial representation, $\pi'(I) = 0$. Then it is clear that $\hat \pi:\Hp \to B$ defined by $\pi'$ is an algebra morphism such that $\hat\pi \circ P = \pi'$, where $P:T(H) \to \Hp$ is the canonical projection.
	
	Thus we have that $\hat{\pi}$ is the unique linear map such that
	\[
        \hat\pi \circ P (x^1 \otimes \cdots \otimes x^n) = \pi'(x^1 \otimes \cdots \otimes x^n),
        \]
	so
	\begin{align*}
	\hat\pi[x^1] \cdots \hat\pi[x^n]
	&= \hat\pi([x^1] \cdots [x^n])\\
	&= \hat\pi(P(x^1) \cdots P(x^n))\\
	&= \hat\pi \circ P(x^1 \otimes \cdots \otimes x^n)\\
	&= \pi' (x^1 \otimes \cdots \otimes x^n)\\
	&= \pi' (x^1) \cdots \pi' (x^n)\\
	&= \pi (x^1) \cdots \pi (x^n).
	\end{align*}
	
	Since $\pi(1_H) = 1_B$, it follows that $\hat\pi$ is the unique algebra morphism such that $\hat\pi \circ [~] = \pi$.
	
	Conversely, suppose that there exists $\phi: \Hp \to B$ an algebra morphism, so clearly $\phi \circ [~]: H \to B$ is a partial representation of $H$. Then, by the above shown, there is a unique $\pi_\phi: \Hp \to B$ such that $\pi_\phi \circ [~] = \phi \circ [~]$.
	
	From the uniqueness of $\pi_\phi$, we have that $\phi = \pi_\phi$.
\end{proof}

 \begin{ex}\label{ex_groupoid}
 Let $\G$ be a groupoid. The domain and the range of an element $g \in \G$ are denoted by $d(g)$ and $r(g)$, respectively, and $\G_2$ is the set of composible pairs of $\G$.  For the case in which the weak Hopf algebra is the groupoid algebra $\Bbbk \G$, $\Hp$ is the algebra generated by the Birget-Rhodes expansion $\G^{BR}$ of $\G$ \cite[Corollary 3.5]{lata2}, where defining $Y_g = \{ h \in \G : r(h) = d(g) \}$, we have $\G^{BR} = \{ (A,g) : d(g), g^{-1} \in A \subseteq Y_g\}$, which is a groupoid with partial multiplication given by
    \begin{align*}
        (A,g) \cdot (B,h) = \begin{cases}
            (B,gh), \text{ if } (g,h) \in \G_2 \text{ and } A = hB, \\
            \text{undefined, otherwise.}
        \end{cases}
    \end{align*}

 \end{ex}

Now define $\E_h = [h_1][S(h_2)]$ and $\Et_h = [S(h_1)][h_2]$ for each $h \in H$. We have some properties for these elements.

\begin{prop}\label{propee'}
	For any $h,k \in H$  we have:
	\begin{enumerate}[(a)]
		\item $\E_k [S(h)] = [S(h_1)]\E_{h_2 k}$;\label{propee'-1}
		\item $[h]\E_k = \E_{h_1k}[h_2] $;\label{propee'-2}
		\item $\E_{h_1}\E_{h_2} = \E_h $;\label{propee'-3}
		\item $\Et_k [h] = [h_1]\Et_{kh_2} $;\label{propee'-4}
		\item $[S(h)]\Et_k = \Et_{k h_1}[S(h_2)] $;\label{propee'-5}
		\item $\Et_{h_1}\Et_{h_2} = \Et_h  $;\label{propee'-6}
		\item $[h] \E_k = [h_1]\E_k\Et_{h_2} $;\label{propee'-7}
		\item $\E_k [S(h)] = \Et_{h_1}\E_k[S(h_2)] $;\label{propee'-8}
		\item $\Et_k [h] = \E_{h_1}\Et_k[h_2] $;\label{propee'-9}
		\item $[S(h)] \Et_k = [S(h_1)]\Et_k\E_{h_2} $;\label{propee'-10}
		\item $\E_h\Et_k = \Et_k\E_h $;\label{propee'-11}
		\item $\Et_{kS^{-1}(h_3)h_1}\Et_{h_2} = \Et_{h}\Et_{k}$; \label{propee'-12}
		\item  $\E_{h_1S(h_3)}\E_{h_2} = \E_h$; \label{propee'-13}
  \item If $H$ is cocommutative, $E_hE_k=E_kE_h$. \label{propee'-14}
	\end{enumerate}
 \end{prop}

\begin{proof}
	For \ref{propee'-1} 
	\begin{align*}
	[S(h_1)] \E_{h_2k}
	&= [S(h_1)][h_2 k_1][S(h_3 k_2)] \\
	&= [S(h_1) h_2 k_1][S(h_3 k_2)] \\
	\OSeq{p4-2}&= [1_1k_1][S(h 1_2 k_2)] \\
	&= [k_1][S(h k_2)] \\
	&= [k_1][S(k_2) S(h)] \\
	&= [k_1][S(k_2)][S(h)] \\
	&= \E_k[S(h)].
	\end{align*}
	
	For \ref{propee'-2}
	\begin{align*}
	\E_{h_1k}[h_2]
	&= [h_1k_1][S(h_2 k_2)][h_3]\\
	&= [h_1k_1][S(h_2 k_2)h_3] \\
	&= [h_1k_1][S(k_2)S(h_2)h_3]\\
	\OSeq{p12-1}&= [h 1_1k_1][S(k_2) S(1_2)] \\
	&= [h k_1][S(k_2)] \\
	&= [h][k_1][S(k_2)] \\
	&= [h] \E_k.
	\end{align*}
	
	For \ref{propee'-3}
	\begin{align*}
	\E_{h_1}\E_{h_2}
	&= [h_1][S(h_2)][h_3][S(h_4)] \\
	&= [h_1][S(h_2)]\\
	&= \E_h.
	\end{align*}
	
	For \ref{propee'-4}
	\begin{align*}
	[h_1]\Et_{kh_2}
	&= [h_1][S(k_1h_2)][k_2h_3]\\
	&= [h_1S(k_1 h_2)][k_2h_3]\\
	&= [h_1S(h_2)S(k_1)][k_2h_3]\\
	\OSeq{p12-2}&= [S(1_1)S(k_1)][k_2 1_2 h] \\
	&= [S(k_1 1_1)][k_2 1_2 h] \\
	&= [S(k_1)][k_2 h] \\
	&= [S(k_1)][k_2][h] \\
	&= \Et_k [h].
	\end{align*}
	
	For \ref{propee'-5}
	\begin{align*}
	\Et_{k h_1}[S(h_2)]
	&= [S(k_1 h_1)][k_2 h_2][S(h_3)]\\
	&= [S(k_1 h_1)][k_2 h_2 S(h_3)]\\
	\OSeq{p4-1}&= [S(k_1 1_1 h)][k_2 1_2]\\
	&= [S(k_1 h)][k_2]\\
	&= [S(h) S(k_1)][k_2]\\
	&= [S(h)][S(k_1)][k_2]\\
	&= [S(h)] \Et_k.
	\end{align*}
	
	For \ref{propee'-6}
	\begin{align*}
	\Et_{h_1}\Et_{h_2}
	&= [S(h_1)][h_2][S(h_3)][h_4]\\
	&= [S(h_1)][h_2] \\
	&= \Et_h.
	\end{align*}
	
	For \ref{propee'-7}
	\begin{align*}
	[h]\E_k
	\OSeq{propee'-2}&= [h_1]\E_{h_1k}[h_2]\\
	&= [h_1]\E_{h_1k}[h_2][S(h_3)][h_4]\\
	\OSeq{propee'-2}&= [h_1]\E_k[S(h_2)][h_3]\\
	&= [h_1]\E_k \Et_{h_2}.
	\end{align*}
	
	For \ref{propee'-8}
	\begin{align*}
	\E_k [S(h)]
	\OSeq{propee'-1}&= [S(h_1)]\E_{h_2 k} \\
	&= [S(h_1)][h_2][S(h_3)]\E_{h_4 k}\\
	\OSeq{propee'-1}&= [S(h_1)][h_2]\E_k [S(h_3)]\\
	&= \Et_{h_1} \E_k [S(h_2)].
	\end{align*}
	
	For \ref{propee'-9}
	\begin{align*}
	\Et_k [h]
	\OSeq{propee'-4}&= [h_1]\Et_{kh_2} \\
	&= [h_1][S(h_2)][h_3]\Et_{kh_4} \\
	\OSeq{propee'-4}&= [h_1][S(h_2)]\Et_k[h_3] \\
	&= \E_{h_1}\Et_k[h_3].
	\end{align*}
	
	For \ref{propee'-10}
	\begin{align*}
	[S(h)]\Et_k
	\OSeq{propee'-5}&= \Et_{k h_1}[S(h_2)]\\
	&= \Et_{k h_1}[S(h_2)][h_3][S(h_4)]\\
	\OSeq{propee'-5}&= [S(h_1)]\Et_k[h_2][S(h_3)]\\
	&= [S(h_1)]\Et_k \E_{h_2}.
	\end{align*}
	
	For \ref{propee'-11}
	\begin{align*}
	\E_h \Et_k
	&= [h_1][S(h_2)] \Et_k \\
	\OSeq{propee'-5}&= [h_1] \Et_{k h_2}[S(h_3)] \\
	&= [h_1] \Et_{k h_2}[S(h_3)][h_4][S(h_5)] \\
	\OSeq{propee'-5}&= \E_{h_1} \Et_k [h_2][S(h_3)] \\
	\OSeq{propee'-9}&= \Et_k [h_1][S(h_2)] \\
	&= \Et_k \E_h.
	\end{align*}
	
	For \ref{propee'-12}
	\begin{align*}
	\Et_{kS^{-1}(h_3)h_1}\Et_{h_2}
	&= [S(h_1)h_6S(k_1)][k_2S^{-1}(h_5)h_2][S(h_3)][h_4]\\
	&= [S(h_1)S((kS^{-1}(h_3))_1)][(kS^{-1}(h_3))_2][h_2]\\
	&= [S(h_1)][S((kS^{-1}(h_3))_1)][(kS^{-1}(h_3))_2][h_2]\\
	&= [S(h_1)][h_4S(k_1)][k_2S^{-1}(h_3)h_2]\\
	&= [S(h_1)][h_21_2S(k_1)][k_2S^{-1}(1_1)]\\
	&= [S(h_1)][h_2S(1_1)S(k_1)][k_21_2]\\
	&= \Et_{h}\Et_{k}
	\end{align*}
	
	For \ref{propee'-13}
	\begin{align*}
	\E_{h_1S(h_3)}\E_{h_2}
	&= [h_1S(h_6)][S^2(h_5)S(h_2)][h_3][S(h_4)] \\
	&= [h_1S(h_4)][SS(h_3)][S(h_2)]\\
	&= [h_1S(h_3)][S(S(h_2)_1][S(h_2)_2]\\
	&= [h_1][S(h_2)_1][S(S(h_2)_2)][S(h_2)_3]\\
	&= [h_1][S(h_2)] = \E_h.
	\end{align*}
    For \ref{propee'-14}
    \begin{align*}
    E_{h}E_{k} 
    &= \left[h_{1}\right]\left[S\left(h_{2}\right)\right] E_{k} \\
    \OSeq{propee'-2}&=\left[h_{1}\right] E_{S\left(h_{3}\right) k}\left[S\left(h_{2}\right)\right] \\
    \OSeq{propee'-2}&=E_{h_{1} S\left(h_{4}\right) k}\left[h_{2}\right]\left[S\left(h_{3}\right)\right] \\ 
    \OS{*}&= E_{h_{2} S\left(h_{3}\right) k} \left[h_{1}\right]\left[S\left(h_{4}\right)\right]  \\
    \OSeq{p4-1}&= E_{1_2k} \left[1_1h_{1}\right]\left[S\left(h_{2}\right)\right] \\
    \OS{*}&= E_{1_1k} \left[1_2h_{1}\right]\left[S\left(h_{2}\right)\right] \\
    \OSeq{defprep-2}&= E_{1_1k}[1_2]\left[h_{1}\right]\left[S\left(h_{2}\right)\right]\\
    \OSeq{propee'-2}&= [1_H]E_kE_h = E_{k}E_{h}
    \end{align*}
    where $\ast$ is for cocommutativity.
\end{proof}

\begin{prop}
	If $H$ has a bijective antipode then we obtain
	\begin{enumerate}[(a)]\label{propee'si}
		\item $\E_k [h] = [h_2]\E_{\Si(h_1)k}$;\label{propee'si-1}
		\item $[h]\Et_k = \Et_{k\Si(h_2)}[h_1] $;\label{propee'si-2}
		\item $\E_k [h] = \Et_{\Si(h_2)}\E_k[h_1] $;\label{propee'si-3}
		\item $[h] \Et_k = [h_2]\Et_k\E_{\Si(h_1)} $;\label{propee'si-4}
	\end{enumerate}
\end{prop}

\begin{proof}
	It is straightforward using $\Si(h)$ in items \ref{propee'-1}, \ref{propee'-5}, \ref{propee'-8} and \ref{propee'-10} of \autoref{propee'}.
\end{proof}

Consider $\A$ defined as the subalgebra of $\Hp$ generated by the elements $\E_h$, that is, 
\[
\A = \langle \E_h \in \Hp |~ h\in H \rangle.
\]

We aim to show that $\A$ is a symmetric partial $H$-module algebra and then construct the partial smash product $\A \Sp H$.

Since
\[[h_1]\E_k[S(h_2)] \overset{\ref{propee'-2}}= \E_{h_1 k} [h_2] [S(h_3)] = \E_{h_1 k} \E_{h_2}\]
then the following linear map
\begin{align*}
\Ape: H \otimes \A & \longto \A\\
h \otimes a & \longmapsto [h_1] a [S(h_2)]
\end{align*}
is clearly well defined.

\begin{prop}
	With the notations as defined above, $\A$ is a symmetric partial $H$-module algebra with partial action given by $\Ape$.
\end{prop}
\begin{proof}
	First of all, notice that from \autoref{propee'}-\ref{propee'-11}, for any $a\in \A$ and any $h\in H$, it follows that
	
	\begin{equation}
	   \Et\!_h ~ a = a ~ \Et\!_h. \label{e'a=ae'}
	\end{equation}

	Let $a$ be an element in $\A$. Hence
	\begin{align*}
	1_H \Ape a
	&= [1_1] a [S(1_2)]\\
	&= [1_1] \Et_{1_2} a [S(1_3)]\\
	\OSeq{e'a=ae'}&= [1_1] a \Et_{1_2} [S(1_3)]\\
	&= [1_1] a [S(1_2)][1_3][S(1_4)]\\
	&= [1_1] a [S(1_2)][1_3 1_{1'}][S(1_{2'})]\\
	&= [1_1] a [S(1_2)][1_3] [1_{1'}][S(1_{2'})]\\
	&= [1_1] a \Et_{1_2}\\
	\OSeq{e'a=ae'}&= [1_1] \Et_{1_2} a\\
	&= a
	\end{align*}
	
	Let $a$ and $b$ be two elements in $\A$ and $h$ in $H$. Hence
	\begin{align*}
	h \Ape ab
	&= [h_1] ab [S(h_2)]\\
	\OSeq{defprep-6}&= [h_1] \Et_{h_2} ab [S(h_3)]\\
	\OSeq{e'a=ae'}&= [h_1] a \Et_{h_2} b [S(h_3)]\\
	&= [h_1] a [S(h_2)][h_3] b [S(h_4)]\\
	&= (h_1 \Ape a) (h_2 \Ape b).
	\end{align*}
	
	Now let $h,k \in H$ and $a\in \A$. Then
	\begin{align*}
	h \Ape (k \Ape a)
	&= [h_1][k_1]a[S(k_2)][S(h_2)]\\
	\OSeq{defprep-6}&= [h_1][k_1]\Et_{k_2}a[S(k_3)][S(h_2)]\\
	\OSeq{e'a=ae'}&= [h_1][k_1]a\Et_{k_2}[S(k_3)][S(h_2)]\\
	&= [h_1][k_1]a\Et_{k_2}[S(k_3)S(h_2)]\\
	\OSeq{e'a=ae'}&= [h_1][k_1]\Et_{k_2}a[S(h_2 k_2)]\\
	\OSeq{defprep-6}&= [h_1]\Et_{h_2}[k_1]a[S(h_3 k_2)]\\
	&= [h_1][S(h_2)][h_3][k_1]a[S(h_4 k_2)]\\
	&= [h_1][S(h_2)][h_3 k_1]a[S(h_4 k_2)]\\
	&= (h_1 \Ape 1_A)(h_2 k \Ape a).
	\end{align*}
	Moreover,
	\begin{align*}
	h \Ape (k \Ape a)
	&= [h_1][k_1]a[S(k_2)][S(h_2)]\\
	\OSeq{defprep-6}&= [h_1][k_1]\Et_{k_2}a[S(k_3)][S(h_2)]\\
	&= [h_1k_1]\Et_{k_2}a[S(k_3)][S(h_2)]\\
	\OSeq{e'a=ae'}&= [h_1k_1]a\Et_{k_2}[S(k_3)][S(h_2)]\\
	\OS{\ref{6-equiv}\eqref{6-equiv-4}}&= [h_1k_1]a[S(k_2)][S(h_2)]\\
	\OS{\ref{6-equiv}\eqref{6-equiv-4}}&= [h_1k_1]a[S(k_2)][S(h_2)][h_3][S(h_4)]\\
	&= [h_1k_1]a[S(k_2)S(h_2)][h_3][S(h_4)]\\
	&= [h_1k_1]a[S(h_2 k_2)][h_3][S(h_4)]\\
	&= (h_1k \Ape a)(h_2\Ape 1_A).
	\end{align*}
	And it concludes the statement.
\end{proof}

Since $\A$ is a symmetric partial $H$-module algebra, we have that the partial smash product $\A \Sp H$ that is a unital algebra (c.f. \cite{CPQS}*{Section~6}).

With the above construction we have the following result.

\begin{teo}
	Let $H$ be a weak Hopf algebra with invertible antipode $S$. Then the universal algebra $\Hp$ is isomorphic to $\A \Sp H$, as algebras.
\end{teo}
\begin{proof}
	Since $A$ is a symmetric partial $H$-module algebra, the map
	\begin{align*}
	\pi_\A: H & \longto \A \Sp H\\
	h & \longmapsto 1_\A \Sp h
	\end{align*}
	is a partial representation. Then, by the universal property of $\Hp$, we have that there exists a unique algebra map
	\[\hat\pi: \Hp \longto \A \Sp H\]
	such that $\hat\pi \circ [~] = \pi_\A$, so $\hat\pi[h] = 1_\A \Sp h$.
	
	On the other hand, the inclusion map $\imath:\A \to \Hp$ is an algebra map and $[~]:H \to \Hp$ is a partial representation. Moreover, by the action defined in $\A$ and since $\E_h\Et_k = \Et_k\E_h$, it follows that $(\imath,[~])$ is a covariant pair.
	
	Therefore, by the universal property of the partial smash there exists a unique algebra morphism
	\begin{align*}
	\Phi\colon \A \Sp H & \longto \Hp\\
	a \Sp h & \longmapsto \imath(a) [h] = a[h].
	\end{align*}
	
	Then, $\Phi(\hat\pi([h])) = \Phi(1 \Sp h) = 1_\A [h] = [h]$. To calculate the another composition, we will make the computation for a basic element $\E_g \in \A$
	\begin{align*}
	\hat\pi(\Phi (\E_g \Sp 1_H))
	&= \hat\pi (\E_g)\\
	&= \hat\pi[g_1]\hat\pi[S(g_2)]\\
	&= (1 \Sp g_1)(1 \Sp S(g_2))\\
	&= (g_1 \Ape 1) \Sp g_2 S(g_3)\\
	\OSeq{p4-1}&= (1_1 g \Ape 1) \Sp 1_2\\
	\OSeq{lemaadicional}&= (1_1 \Ape (g \Ape 1)) \Sp 1_2\\
	&= (1 \Sp 1)(g \Ape 1_\A \Sp 1_H)\\
	&= (g \Ape 1_\A) \Sp 1_H\\
	&= [g_1][S(g_2)]\Sp 1_H\\
	&= \E_g\Sp 1_H.
	\end{align*}
	
	Now if $a = \prod_{i=1}^{n} \E_{h^i} \in \A$, so
	\begin{align*}
	\hat\pi(\Phi (a \Sp h))
	&= \hat\pi(\Phi \left(\prod_i \E_{h^i} \Sp h \right) )\\
	\OSeq{quebradosmash}&= \hat\pi(\Phi \left( \prod_i \E_{h^i} \Sp 1 \right) (1_\A \Sp h) )\\
	\OSeq{imersaonosmash}&= \hat\pi(\Phi \left( \prod_i \left( \E_{h^i} \Sp 1 \right) (1_\A \Sp h) \right))\\
	&= \prod_i \hat\pi(\Phi \left( \E_{h^i} \Sp 1 \right))  \hat\pi(\Phi (1_\A \Sp h))\\
	&= \prod_i \left( \E_{h^i} \Sp 1 \right) \hat\pi[h]\\
	\OSeq{imersaonosmash}&= \left( \prod_i \E_{h^i} \Sp 1 \right) (1_\A \Sp h)\\
	&= (a \Sp 1)(1_\A \Sp h)\\
	\OSeq{quebradosmash}&= a \Sp h.
	\end{align*}
	
	Therefore, $\Hp$ and $\A \Sp H$ are isomorphic as algebras.
\end{proof}

\section{\texorpdfstring{$\Hp$}{H\scriptsize par} is a Hopf Algebroid}
\def\op{{\mathop{{}^{\text{op}}}}}
\def\copop{{\mathop{{}^{\text{cop~op}}}}}
\def\S{\mathop{\mathcal{S}}\mathop{}}

Given a weak Hopf algebra with invertible antipode, as well as in the ``strong" case, one can prove that $\Hp$ is a Hopf algebroid. We start recalling the definition of Hopf algebroid.

\begin{dfn}\label{algebroid}
\cite{bohm} Given a $\Bbbk$-algebra $A$, a left (resp. right) bialgebroid over $A$ is
given by the data $(\Ha, A, s, t, \underline{\Delta}_l, \underline{\epsilon}_l)$ (resp. $(\Ha, A, \tilde{s}, \tilde{t}, \underline{\Delta}_r, \underline{\epsilon}_r)$ such that:
\begin{enumerate}
    \item $\Ha$ is a $\Bbbk$-algebra;
    \item The map $s$ (resp. $\tilde{s}$) is a morphism of algebras between $A$ and $\Ha$, and the map $t$ (resp. $\tilde{t}$) is an anti-morphism of algebras between $A$ and $\Ha$. Their images commute. By the maps $s, t$ (resp. $\tilde{s}$, $\tilde{t}$) the algebra $\Ha$
inherits a structure of $A$-bimodule given by $a \triangleright h \triangleleft b = s(a)t(b)h$ (resp. $a \triangleright h \triangleleft b = h\tilde{s}(a)\tilde{t}(b))$.
    \item The triple $(\Ha, \underline{\Delta}_l, \underline{\epsilon}_l)$ (resp. $(\Ha, \underline{\Delta}_r, \underline{\epsilon}_r)$ is an $A$-coring relative to the structure of $A$-bimodule defined by $s$ and $t$ (resp. $\tilde{s}$ and $\tilde{t}$).
    \item The image of $\underline{\Delta}_l$ (resp. $\underline{\Delta}_r$) lies on the Takeuchi subalgebra $\Ha \times_A \Ha = \{\sum_i h_i \otimes k_i \in \Ha \otimes_A \Ha \mid \sum_i h_it(a) \otimes k_i = \sum_i h_i \otimes k_is(a)\}$ (resp. $\Ha {}_A\times \Ha = \{\sum_i h_i \otimes k_i \in \Ha \otimes_A \Ha \mid \sum_i \tilde{s}(a)h_i \otimes k_i = \sum_i h_i \otimes \tilde{t}(a)k_i\}$), and it is an algebra morphism.
    \item For every $h, k \in \Ha$, we have $\underline{\epsilon}_l(hk) = \underline{\epsilon}_l(hs(\underline{\epsilon}_l(k))) = \underline{\epsilon}_l(ht(\underline{\epsilon}_l(k)))$ (resp. $\underline{\epsilon}_r(hk) = \underline{\epsilon}_r(\tilde{s}(\underline{\epsilon}_r(h))k) = \underline{\epsilon}_r(\tilde{t}(\underline{\epsilon}_r(h))k))$. 
\end{enumerate}
Given two anti-isomorphic algebras $A$ and $\tilde{A}$ (ie, $A \cong \tilde{A}^{op}$), $A$ a left bialgebroid $(\Ha, A, s, t, \underline{\Delta}_l$, $\underline{\epsilon}_l)$ and $\tilde{A}$ a right bialgebroid $(\Ha, \tilde{A}, \tilde{s}, \tilde{t}, \underline{\Delta}_r, \underline{\epsilon}_r)$, a Hopf algebroid structure on $\Ha$ is given by an antipode, that is, an anti algebra homomorphism $\Sa: \Ha \longto \Ha$, such that 
\begin{enumerate}[(i)]
    \item $s \circ \underline{\epsilon}_l \circ \tilde{t} = \tilde{t}$, $t \circ \underline{\epsilon}_l \circ \tilde{s} = \tilde{s}$, $\tilde{s} \circ \underline{\epsilon}_r \circ t = t$ and $\tilde{t} \circ \underline{\epsilon}_r \circ s = s$.
    \item $(\underline{\Delta}_l \otimes_{\tilde{A}} I) \circ \underline{\Delta}_r = (I \otimes_A \underline{\Delta}_r) \circ \underline{\Delta}_l$ and $(I \otimes_{\tilde{A}} \underline{\Delta}_l) \circ \underline{\Delta}_r = (\underline{\Delta}_r \otimes_A I) \circ \underline{\Delta}_l$.
    \item $\Sa(t(a)h\tilde{t}(\tilde{b})) = \tilde{s}(\tilde{b})\Sa(h)s(a),$ for all $a \in A$, $\tilde{b} \in \tilde{A}$ and $h \in \Ha$.
    \item $\mu_{\Ha} \circ (\Sa \otimes I) \circ \underline{\Delta}_l = \tilde{s} \circ \underline{\epsilon}_r$ and $\mu_{\Ha} \circ (I \otimes \Sa) \circ \underline{\Delta}_r = s \circ \underline{\epsilon}_l$.
\end{enumerate}
\end{dfn}

Now we can define the structure of Hopf algebroid of $\Hp$. Notice that the antipode $S: H \longto H$ and its inverse define two anti-algebra maps: $\Sa: \Hp \longto \Hp$ and $\Sa': \Hp \longto \Hp$, given, respectively, by
\begin{align*}
\Sa([h^1] \cdots [h^n]) &  = [S(h^n)] \cdots [S(h^1)]\\ 
\Sa'([h^1] \cdots [h^n]) & = [S^{-1}(h^n)] \cdots [S^{-1}(h^1)].
\end{align*}

The maps $\Sa$ and $\Sa'$ show that the subalgebras $\A = \langle \E_h \in \Hp |~ h\in H \rangle$ and $\At := \langle \Et_h \in \Hp |~ h\in H \rangle$ of $\Hp$ are anti-isomorphic algebras and by \autoref{propee'}\ref{propee'-11} they commute one with each other. Now we can define the following source and target maps:
\begin{align*}
s,t: \A \longto \Hp, \quad s(a) = a \text{ and } t(a) = \Sa'(a), \forall a \in \A  \\
\tilde{s}, \tilde{t}: \At \longto \Hp, \quad \tilde{s}(\tilde{a}) = \tilde{a} \text{ and } \tilde{t}(\tilde{a}) = \Sa'(\tilde{a}), \forall \tilde{a} \in \At.
\end{align*}

Also, $\Hp$ has two bialgebroid structures: the left one, relative to the base algebra $\A$, inherited by $s,t$, and the right one, relative to the base algebra $\At$, inherited by $\tilde{s},\tilde{t}$.

For the left part, the comultiplication and the counit are given by

\begin{align*}
\underline{\Delta}_l([h^1] \cdots [h^n]) = [h^1_{1}] \cdots [h^n_{1}] \otimes_A [h^1_{2}] \cdots [h^n_{2}]\\
\underline{\epsilon}_l([h^1] \cdots [h^n]) = E_{h^1_{1}}E_{h^1_{2}h^2_{1}} \cdots E_{h^1\cdots h^{n}}
\end{align*}

For the right part, the comultiplication and the counit are given by
\begin{align*}
\underline{\Delta}_r([h^1] \cdots [h^n]) = [h^1_{1}] \cdots [h^n_{1}]\otimes_{\tilde{A}} [h^1_{2}] \cdots [h^n_{2}]\\
\underline{\epsilon}_r([h^1] \cdots [h^n]) = \tilde{E}_{h^1_{1}h^2_{1}\cdots h^n_{1}}\tilde{E}_{h^2_{2}\cdots h^n_{2}} \cdots \tilde{E}_{h^{n}}
\end{align*}

Now we can prove that $\Hp$ is a Hopf algebroid.

\begin{teo} \label{hopfalg}
    Let $H$ be a weak Hopf algebra with invertible antipode. Then the data $(\Hp$, $\A$, $\At$, $s$, $t$, $\tilde{s}$, $\tilde{t}$, $\underline{\Delta}_l$, $\underline{\Delta}_r$, $\underline{\epsilon}_l$, $\underline{\epsilon}_r)$ define a Hopf algebroid structure on $\Hp$.
\end{teo}

\begin{proof} We shall remark that in some steps of the proof we will take an element $[h] \in \Hp$ or $[h][k] \in \Hp$ instead of $[h^1]\cdots[h^n] \in \Hp$, whenever the general case follows similarly to the particular case.

Now, (1) and (2) of \autoref{algebroid} easily hold.

(3) We shall verify that the triple $(\Hp, \underline{\Delta}_l, \underline{\epsilon}_l)$ is an $\A$-coring relative to the structure of $\A$-bimodule defined by $s$ and $t$. Thus we must prove that $\underline{\Delta}_l$ and $\underline{\epsilon}_l$ are morphisms of $\A$-bimodules. It follows by construction that $\underline{\epsilon}_l$ is a left $\A$-linear map. We will check that $\underline{\epsilon}_l$ is a right $\A$-linear map. The $\A$-bilinearity of $\underline{\Delta}_l$ is left to the reader.

\begin{align*}
\underline{\epsilon}_l([h] \Aid E_k)
& = \underline{\epsilon}_l(\tilde{E}_{S^{-1}(k)}[h]) = \underline{\epsilon}_l([k_2][S^{-1}(k_1)][h])\\
& = \underline{\epsilon}_l([k_2][S^{-1}(k_1)h]) = \underline{\epsilon}_l([k_2][S(k_3)][k_4][S^{-1}(k_1)h])\\
& = \underline{\epsilon}_l(E_{k_2}[k_3S^{-1}(k_1)h]) = E_{k_2}E_{k_3S^{-1}(k_1)h}\\
& = [k_3][S(k_4)][k_5S^{-1}(k_2)h_1][S(k_6S^{-1}(k_1)h_2]\\
& =  [k_3][S^{-1}(k_2)h_1][S(S^{-1}(k_1)h_2)S(k_4)]\\
& = [k_3S^{-1}(k_2)h_1][S(h_2)k_1][S(k_4)]\\
\OSeq{p4-1}&= [S^{-1}(1_2)h_1][S(h_2)1_1k_1][S(k_2)]\\
& = [S^{-1}(S(1_1))h_1][S(h_2)S(1_2)k_1][S(k_2)]\\
& = [1_1h_1][S(1_2h_2)k_1][S(k_2)]\\
& = [h_1][S(h_2)][k_1][S(k_2)] \\
& = E_hE_k = \underline{\epsilon}_l([h])E_k.
\end{align*}

Also, the coassociativity of $\underline{\Delta}_l$ easily follows by its definition. Let us check the counit property. Indeed, for $[h] \in \Hp$, 

\begin{align*}
\mu_{\Aid} \circ (I_{\Hp} \otimes_{\A}  \underline{\epsilon}_l) \circ \underline{\Delta}_l([h])
& = [h_1] \Aid \underline{\epsilon}_l([h_2]) = [h_1] \Aid E_{h_2} \\
& = t(E_{h_2})[h_1] = \tilde{E}_{S^{-1}(h_2)}[h_1] = [h_3][S^{-1}(h_2)][h_1] = [h],
\end{align*}

\[
\mu_{\Aie} \circ (\underline{\epsilon}_l \otimes_{\A}  I_{\Hp}) \circ \underline{\Delta}_l([h])
= \underline{\epsilon}_l([h_1]) \Aie [h_2] = E_{h_1}[h_2] = [h].
\]

(4) We shall prove that the image of $\underline{\Delta}_l$ lies on the Takeuchi subalgebra $\Hp \times_{\A} \Hp = \{\sum_i [h_i] \otimes [k_i] \in \Hp \otimes_{\A} \Hp \mid \sum_i [h_i]t(a) \otimes [k_i] = \sum_i [h_i] \otimes [k_i]s(a), \forall a \in \A\}$. Indeed, for $[h] \in \Hp$ and $E_k \in \A$,
\begin{align*}
[h_1] \otimes_{\A} [h_2]s(E_k)
& = [h_1] \otimes_{\A} [h_2][k_1][S(k_2)] =  [h_1] \otimes_{\A} [h_2k_1][S(k_2)]\\
& = [h_1]  \otimes_{\A} E_{h_2k_1}[h_3k_2][S(k_3)] = [h_1]  \otimes_{\A} E_{h_2k_1}[h_3k_2S(k_3)]\\
& = [h_3k_2][S^{-1}(h_2k_1)][h_1] \otimes_{\A} [h_4]\\
& = [h_3k_2][S^{-1}(k_1)S^{-1}(h_2)h_1] \otimes_{\A} [h_4]\\
& = [h_1k_2][S^{-1}(k_1)] \otimes_{\A} [h_2] = [h_1][k_2][S^{-1}(k_1)] \otimes_{\A} [h_2] \\
& = [h_1]\Sa'([k_1S(k_2)]) \otimes_{\A} [h_2] = [h_1]t(E_k) \otimes_{\A} [h_2].
\end{align*}

(5) First notice that
\begin{equation}\label{obsE}
\begin{split}
E_{h_1k}E_{h_2}
    & = [h_1k_1][S(k_2)S(h_2)][h_3][S(h_4)] = [h_1k_1][S(k_2)][S(h_2)]\\
    & = [h_1][k_1][S(h_2k_2)] = [h_1][S(h_2)][h_3][k_1][S(h_2k_2)]= E_{h_1}E_{h_2k}.
\end{split}
\end{equation}

Given \(a \in \A\) and $[h] \in \Hp$, we have that
\begin{align*}
    \underline{\epsilon}_l([h]s(a))
        &= \underline{\epsilon}_l([h_1]\tilde{E}_{h_2}s(a)) =  \underline{\epsilon}_l([h_1]s(a)\tilde{E}_{h_2})\\
        &=  \underline{\epsilon}_l([h_1]s(a)[S(h_2)][h_3])=   \underline{\epsilon}_l(([h_1]s(a)[S(h_2)]) \triangleright [h_3])\\
        &= [h_1]s(a)[S(h_2)] \underline{\epsilon}_l([h_3])= [h_1]s(a)[S(h_2)] E_{h_3}\\
        &= [h_1]s(a)[S(h_2)]= [h_1]a[S(h_2)].
\end{align*}

Now, take $x = [h], y = [k]\in \Hp$. We have
\begin{align*}
\underline{\epsilon}_l(xs(\underline{\epsilon}_l(y)))
    &= \underline{\epsilon}_l([h]s(\underline{\epsilon}_l([k]))) = [h_1]\underline{\epsilon}_l([k])[S(h_2)]\\
    & = [h_1]E_{k}[S(h_2)] \overset{\ref{propee'-2}}= E_{h_1k}[h_2][S(h_2)]  = E_{h_1k}E_{h_2}\\
    & \overset{\eqref{obsE}}{=} E_{h_1}E_{h_2k} = \underline{\epsilon}_l([h][k]).
\end{align*}

The case \(y = [k^1]\dots[k^n]\) follows by \eqref{obsE}.

(i) Let $\tilde{a} = \Et_h \in \At$. Then
\begin{align*}
s \circ \underline{\epsilon}_l \circ \tilde{t} (\tilde{a})
    & = s(\underline{\epsilon}_l(\tilde{t}(\tilde{a}))) \\
    & = s(\underline{\epsilon}_l([S^{-1}(h_2)][h_1])) = E_{S^{-1}(h_3)}E_{S^{-1}(h_2)h_1}\\
    & = [S^{-1}(h_6)][h_5][S^{-1}(h_4)h_1][S(S^{-1}(h_3)h_2)]\\
    & = [S^{-1}(h_6)][h_5][S^{-1}(h_4)][h_1][S(S^{-1}(h_3)h_2)]\\
    & = [S^{-1}(h_4)][h_1][S(h_2)h_3] = [S^{-1}(h_4)][h_1][S(h_2)][h_3]\\
    & = [S^{-1}(h_2)][h_1]  = \tilde{t}(\tilde{a}).
\end{align*}

Also,
\begin{align*}
t \circ \underline{\epsilon}_l \circ \tilde{s} (\tilde{a})
& = t(\underline{\epsilon}_l(\tilde{s}(\tilde{a}))) = t(\underline{\epsilon}_l(\tilde{a}))\\
& = t(\underline{\epsilon}_l(\tilde{E}_h)) = t(\underline{\epsilon}_l([S(h_1)][h_2])) \\
& = t(E_{S(h_2)}E_{S(h_1)h_3}) = \tilde{E}_{S^{-1}(h_3)h_1}\tilde{E}_{h_2} \\
& = [S(h_1)h_6][S^{-1}(h_5)h_2][S(h_3)][h_4]\\
& = [S(h_1)h_4][S^{-1}(h_3)][h_2] = [S(h_1)][h_2] \\
& = \tilde{E}_h = \tilde{s}(\tilde{a}).
\end{align*}

Analogously, we have that $\tilde{s} \circ \underline{\epsilon}_r \circ t = t$ and $\tilde{t} \circ \underline{\epsilon}_r \circ s = s$.

(ii) It is immediate.

(iii) Let $a = E_g \in \A$, $\tilde{b} = \tilde{E}_l \in \At$ and $[h] \in \Hp$.
\begin{align*}
\Sa(t(a)[h]\tilde{t}(\tilde{b}))
& = \Sa(t(E_g)[h]\tilde{t}(\Et_l)) = \Sa(\Sa'(E_g)[h]\Sa'(\Et_l))\\
& = \Sa([g_2][S^{-1}(g_1)][h][S^{-1}(l_2)][l_1]) \\
& = [S(l_1)][l_2][S(h)][g_1][S(g_2)] \\
& = \tilde{s}(\tilde{b})\Sa(h)s(a).
\end{align*}

(iv) Let $x = [h] [k] \in \Hp$.
\begin{align*}
\mu_{\Ha} \circ (I \otimes \Sa) \circ \underline{\Delta}_r(x)
    & = x_1  \Sa(x_2) = [h_{1}] [k_{1}]\Sa ([h_{2}] [k_{2}])\\
    & = [h_{1}] [k_{1}][S(k_{2})] [S(h_{2})] \\
    & = [h_{1}] E_{k}[S(h_{2})] \overset{\ref{propee'-1}}= [h_{1}][S(h_2)]E_{h_3k} \\
    & = E_{h_1}E_{h_2k} \\
    & = \underline{\epsilon}_l([h][k]) =  s \circ \underline{\epsilon}_l(x).
\end{align*}

The case \(y = [k^1]\dots[k^n]\) follows by \autoref{propee'}\ref{propee'-1}. 

Analogously, we have that $\mu_{\Ha} \circ (\Sa \otimes I) \circ \underline{\Delta}_l = \tilde{s} \circ \underline{\epsilon}_r$.
\end{proof}

\begin{obs}\label{Hpar-monoidal}
    Since \(\Hp\) is a Hopf algebroid, we have that \( \left( \CatHparm, \otimes_{\A}, \A \right) \) is a closed monoidal category with forgetful functor
    \[
        \mathcal {U} \colon \CatHparm \longto {}_{\A}\!\!\Cat{Mod}_{\A}
    \]
    monoidal strict that preserves internal homs.
\end{obs}

\section{Partial modules}

This section aims to introduce the notion of a partial module over a weak Hopf algebra, and its relation with the structures developed in this paper. We define partial modules over weak Hopf algebras based on the partial representations of the weak Hopf algebra over the endomorphisms algebra, inspired by the classical connection between modules and representations. As expected, the category of partial modules over a weak Hopf algebra is isomorphic to the category of its partial representations. Moreover, we will see that a symmetric partial module algebra over a weak Hopf algebra corresponds to an algebra object in the module category over the associated universal Hopf algebroid.

Given \(M\) a \(\Bbbk\)-vector space, consider the algebra \(\End_{\Bbbk}(M)\) associated with \(M\). In order to introduce the notion of partial \(H\)-module, we will use the notion of partial representation of \(H\) over \(\End_{\Bbbk}(M)\), and the classical connection between action and representation.

Let \(\bullet \colon H \otimes M \longto M\) be a linear transformation such that the induced map
\[
\begin{alignedat}{3}
    \pi_{\bullet} \colon H & \longto & {}\End_{\Bbbk}(M)\\
    h & \longmapsto & {}\pi_{\bullet}(h) \colon M & \longto M\\
    && m & \longmapsto \pi_{\bullet}(h)(m) = h \bullet m
\end{alignedat}
\]
is a partial representation. In this context, the notion of a partial module arises as follows.

\begin{defin}
    Let \(H\) be a weak Hopf algebra. A \emph{partial \(H\)-module} is a vector space \(M\) with a linear map \(\bullet \colon H \otimes M \longto M\) (called \emph{partial action}) such that the following conditions hold, for any \(m \in M\) and \(h, k \in H\):
    \begin{enumerate}\Not{PM}
        \item \(1_H \bullet m = m\)
        \item \(h \bullet (k_1 \bullet ( S(k_2) \bullet m ) ) = h k_1 \bullet ( S(k_2) \bullet m ) \);
        \item \(h \bullet (S(k_1) \bullet ( k_2 \bullet m ) ) = h S(k_1) \bullet ( k_2 \bullet m ) \);
        \item \(h_1 \bullet (S(h_2) \bullet ( k \bullet m ) ) = h_1 \bullet ( S(h_2) k \bullet m ) \);
        \item \(S(h_1) \bullet (h_2 \bullet ( k \bullet m ) ) = S(h_1) \bullet ( h_2 k \bullet m ) \); and
        \item \(h_1 \bullet (S(h_2) \bullet ( h_3 \bullet m ) ) = h \bullet m \).
    \end{enumerate}
\end{defin}

The following results are immediate.

\begin{prop}
    Given a partial module \((M, \bullet )\), we have that \( \pi_{\bullet } \colon H \longto \End_{\Bbbk }(M) \) defined by \( \pi_{\bullet }(h)(m) = h \bullet m \) is a partial representation.
\end{prop}

\begin{prop}
    Given a partial representation \( \pi \colon H \longto \End_{\Bbbk }(M) \), we have that \(\bullet_{\pi } \colon H \otimes M \longto M \) defined by \(h \bullet_{\pi } m = \pi (h)(m)\) is a partial action.
\end{prop}

\begin{ex}
    Partial \(H\)-module algebras are naturally partial \(H\)-modules.
\end{ex}

\begin{defin}
A homomorphism of partial modules is a linear map respecting the partial actions.
\end{defin}
    
We denote the category of partial \(H\)-modules by \(\Catpm\).

\begin{prop}
    \label{cats-pm-Hparm}
    The categories \(\Catpm\) and \(\CatHparm\) are isomorphic. 
\end{prop}
\begin{proof}
    Given \(M \in \Catpm\) with partial action \(\bullet\), we induce a partial representation \(\pi_{\bullet} \colon H \longto \End_{\Bbbk}(M)\) such that \(\pi_{\bullet}(h)(m) = h \bullet m\).
    
    By the universal property of \(\Hp\), there exists an unique algebra morphism
    \(\hat {\pi}_{\bullet } \colon \Hp\) \(\longto \End_{\Bbbk}(M) \) such that \(\hat {\pi}_{\bullet } ([h]) = \pi_{\bullet }(h)\). So we have that \(M \in \CatHparm\) with action \(\triangleright_{\hat {\pi}_{\bullet}} \colon \Hp \otimes M \longto M\) given by
    \[
        [h] \triangleright_{\hat {\pi}_{\bullet }} m = \hat {\pi}_{\bullet} ([h])(m) = \pi_{\bullet } (h)(m) = h \bullet m.
    \]
    Therefore, we have a functor
    \[
        F \colon \Catpm \Longrightarrow \CatHparm
    \]
    defined by \(F(M, \bullet) = (M, \triangleright_{\hat {\pi}_{\bullet }})\) on objects and by \(F(\varphi) = \varphi\) on morphisms.
    

    On the other hand, given \(M \in \CatHparm\) with action \(\triangleright\), we induce an algebra morphism \(\pi_{\triangleright} \colon \Hp \longto \End_{\Bbbk}(M)\) such that \(\pi_{\triangleright}(h)(m) = h \triangleright m\).
        
    Precomposing with the canonical representation, we get a partial representation \( \tilde {\pi}_{\triangleright } \colon H \longto \End_{\Bbbk}(M) \). So we have that \(M \in \Catpm\) with partial action \(\bullet_{\tilde {\pi}_{\triangleright }} \colon H \otimes M \longto M\) given by
    \[
        h \bullet _{\tilde {\pi}_{\triangleright }} m = \tilde {\pi}_{\triangleright  } (h)(m) = \pi_{\triangleright } ([h])(m) = [h] \triangleright m.
    \]
    Therefore, we have a functor
    \[
        G \colon \CatHparm \Longrightarrow \Catpm
    \]
    defined by \(G(M, \triangleright ) = (M, \bullet _{\tilde {\pi}_{\triangleright }})\) on objects and by \(G(\varphi) = \varphi\) on morphisms.

    Its clear that \(F\) and \(G\) are mutually inverse to each other and, therefore, the categories \( \CatHparm \) and \( \Catpm \) are isomorphic.
\end{proof}

Finally, we will see that the symmetric partial module algebras are algebra objects in $\CatHparm$ and \emph{vice versa}.

\begin{teo}
    \label{spwma=alg-obj-in-HparMod-ida}
    \label{spwma=alg-obj-in-HparMod-volta}
    Let \(H\) be a weak Hopf algebra with bijective antipode. Then there exists an isomorphism between the categories of symmetric partial \(H\)-module algebras and of algebra objects in the category \(\Hp\)-modules, i.e.,
    \[
        _H\Cat{sym\text{-}parMod} \simeq \Cat{Alg} \left( \CatHparm \right). \qedhere
    \]
\end{teo}
\begin{proof}
    Let \((B, \triangleright \colon \Hp \otimes B \longto B) \) be an algebra object in \CatHparm. Hence, there exist \(\mu_B \colon B \otimes_{\A} B \longto B\) and \(\eta_B \colon \A \longto B\) morphisms in \CatHparm\ satisfying the appropriate axioms (associativity and unitality).

    From \autoref{cats-pm-Hparm}, we know that \((B, \bullet ) \) is an object in \Catpm. Since \(\A\) is an algebra over \(\Bbbk\), one can define the linear maps:
    \begin{align*}
        m_B & \colon B \otimes B \longto B \, \text{ by } \,  m_B(b \otimes b') = \mu_B (b \otimes b')\\
        \intertext{and}
        u_B & \colon \Bbbk \longto B \,  \text{ by } \,  u_B(\lambda )  = \eta_B ( \lambda 1_A ) = \lambda 1_B.
    \end{align*}

    The diagrams of associativity and unitality for \( m_B \) and \( u_B \) are commutative, since the diagrams of \( \mu_B \) and \( \eta_B \) are commutative.

    In order to show that \(B\) is a partial \(H\)-module algebra, let us recall that \(B\) is a \(\A\)-bimodule via
    \begin{align}\label{B-eh-Apar-bimod}
        a b &= s(a) \triangleright b = a \triangleright b
        \, \text{ and } \,
        b a = t(a) \triangleright b = \mathcal{S}'(a) \triangleright b,
    \end{align}
    for any \(a \in \A\) and \(b \in B\).
    
    Let \(h, k \in H\) and \(b, b' \in B\). Then:
    \begin{enumerate}
    \item
        Since \(\Hp\) is a Hopf coalgebroid (cf. \autoref{Hpar-monoidal}) and \(\eta_B\) is a morphism in \CatHparm, we have that
        \begin{equation} \label{[h] |> 1_B}
            [h] \triangleright 1_B = [h] \triangleright \eta_B (1_{\A }) = \eta_B (h \cdot 1_{\A }) = \eta_B (E_h) = E_h \eta_B (1_A) = E_h 1_B.
        \end{equation}
    \item\begin{align*}
        h \bullet (b b')
            &= [h] \triangleright (b b')\\
            &= ([h_1] \triangleright b) ([h_2] \triangleright b')\\
            &= (h_1 \bullet b) (h_2 \bullet b')\\
        \end{align*}
    \item \begin{align*}
        h \bullet (k \bullet b)
            &= [h] \triangleright ([k] \triangleright b)\\
            &= [h] [k] \triangleright b\\
            \overset{\eqref{defprep-6}}&= [h_1][S(h_2)][h_3] [k] \triangleright b\\
            \overset{\eqref{defprep-4}}&= [h_1] [S(h_2)] [h_3 k] \triangleright b\\
            &= E_{h_1} [h_2 k] \triangleright b\\
            \overset{\eqref{[h] |> 1_B}}&= ([h_1] \triangleright 1_B) ( [h_2 k] \triangleright b)\\
            &= (h_1 \bullet 1_B) ( h_2 k \bullet b )
        \end{align*}
    \item \begin{align*}
        (h_1k \bullet b)(h_2 \bullet 1_B)
            &= ([h_1 k] \triangleright b)([h_2] \triangleright 1_B)\\
            \overset{\eqref{[h] |> 1_B}}&= ([h_1] \triangleright ( [k] \triangleright b))E_{h_2}\\
           \overset{\eqref{B-eh-Apar-bimod}}&= \mathcal{S}'(E_{h_2}) \triangleright ([h_1] \triangleright ( [k] \triangleright b))\\
           &= [h_3] [S^{-1}(h_2)] \triangleright ([h_1] \triangleright ( [k] \triangleright b))\\
           &= [h_3] [S^{-1}(h_2)] [h_1] \triangleright ( [k] \triangleright b)\\
           &= [S(S^{-1}(h_3))] [S^{-1}(h_2)] [S(S^{-1}(h_1))] \triangleright ( [k] \triangleright b)\\
           &= [S(S^{-1}(h)_1)] [S^{-1}(h)_2] [S(S^{-1}(h)_3)] \triangleright ( [k] \triangleright b)\\
           \overset{\ref{6-equiv}\eqref{6-equiv-4}}&= [S(S^{-1}(h))] \triangleright ( [k] \triangleright b)\\
           &= [h] \triangleright ( [k] \triangleright b)\\
           &= h \bullet ( k \bullet b)
        \end{align*}
    \end{enumerate}

    Therefore, \(B\) is a symmetrical partial \(H\)-module algebra.

    Then, we have the functor
    \[
        F \colon \Cat{Alg} \left( \CatHparm \right) \to {_H}\Cat{sym\text{-}parMod}  
    \]
    given by \(F((B, \triangleright, \mu_B, \eta_B) = (B, \bullet, m_B, u_B)\) on objects and \(F(\varphi) = \varphi\) on morphisms.

    Reciprocally, let \(B\) be a symmetrical partial \(H\)-module algebra via \(\bullet \colon H \otimes B \longto B \). Since \(B\) is an object in \Catpm, it follows that \(B\) is an object in \CatHparm\ with action \(\triangleright \colon \Hp \otimes B \longto B\) given by \([h] \triangleright b \coloneqq h \bullet b\).

    Firstly, we need to endow \(B\) with an structure of \(\A\)-bimodule, as follows:
    \begin{align*}
    E_h b : \mskip-6mu
        &= h_1 \bullet ( S(h_2) \bullet b )\\
        &= (h_1 \bullet 1_B) ( h_2 S(h_3) \bullet b )\\
        \overset{\ref{p4}\eqref{p4-1}}&= (1_1 h \bullet 1_B) ( 1_2 \bullet b )\\
        &= (1_1 h \bullet 1_B) ( 1_2 \bullet 1_B ) ( 1_3 \bullet b )\\
        &= ( 1_1  \bullet ((h \bullet 1_B ) 1_B) ) ( 1_2 \bullet b )\\
        &= 1  \bullet ((h \bullet 1_B ) 1_B b)\\
        &= (h \bullet 1_B ) b\\
    \intertext{and}
    b E_h : \mskip-6mu
        &= h_2 \bullet ( S^{-1}(h_1) \bullet b ) \\
        &= (h_2 S^{-1}(h_1) \bullet b ) (h_3 \bullet 1_B) \\
        &= (S^{-1}(h_1 S(h_2)) \bullet b ) (h_3 \bullet 1_B) \\
        \overset{\ref{p12}\eqref{p12-2}}&= (S^{-1}(S(1_1)) \bullet b ) (1_2 h \bullet 1_B) \\
        &= (1_1 \bullet b ) (1_2 h \bullet 1_B) \\
        &= 1 \bullet (b (h \bullet 1_B)) \\
        &= b (h \bullet 1_B)
    \end{align*}

    Given \(h \in H\) and \(b, b' \in B\), we have:
    \begin{gather*}
        E_h (b b') = (h \bullet 1_B)(b b') = ((h \bullet 1_B)b) b' = (E_h b) b',\\
        (b b') E_h = (b b')(h \bullet 1_B) = b((h \bullet 1_B)b') = b (E_h b'),\\
        b (E_h b') = b((h \bullet 1_B) b') = (b (h \bullet 1_B)) b' = (b E_h) b'\\
    \shortintertext{and}
        [h] \triangleright (b b') = h \bullet (b b') = (h_1 \bullet b ) (h_2 \bullet b') = ([h_1] \triangleright b ) ([h_2] \triangleright b'),
    \end{gather*}
    so the multiplication \(m_B\) of \(B\) as object in \Catpm\ induces a morphism \(\mu_B \colon B \otimes_{\A} B \longto B\) in \CatHparm.

    As \(E_h 1_B = 1_B E_h\), for any \(h \in H\), it follows that \(\eta_B \colon \A \longto B\) given by \(\eta_B(a) \coloneqq a 1_B \) is a morphism in \(_{\A}\Cat{Mod}_{\A}\). Moreover, given \(h \in H\),
    \[
        [h] \triangleright \eta_B(1_A) = h \bullet \eta_B(1_A) = \eta_B( h \Ape 1_A),
    \]
    therefore, \(\eta_B\) is a morphism in \CatHparm.

    Since \((B, m_B, u_B)\) is an algebra, it follows directly that \((B, \mu_B, \eta_B)\) is an algebra object in \CatHparm.

    Then, we have the functor
    \[
        G \colon _H\Cat{sym\text{-}parMod} \to \Cat{Alg} \left( \CatHparm \right)
    \]
    given by \(G (B, \triangleright, m_B, u_B) = (B, \bullet, \mu_B, \eta_B)\) on objects and \(F (\varphi) = \varphi\) on morphisms.

    Its clear that \(F\) and \(G\) are mutually inverse to each other and, therefore, the categories \(_H\Cat{sym\text{-}parMod}\) and \( \Cat{Alg} \left( \CatHparm \right) \) are isomorphic.
\end{proof}

\section{Quantum inverse semigroups}

In \cite{qua}, a new Hopf structure was introduced, called \emph{quantum inverse semigroup}. The primary goal was to extend the concept of inverse semigroups into the Hopf context, analogous to how Hopf algebras generalize groups. In that work, the authors proved that when $H$ is a cocommutative Hopf algebra, the partial Hopf algebra $H_{par}$ has a quantum inverse semigroup structure. When we are dealing with weak Hopf algebras, it also happens.  So, in this section, we shall prove that $\Hp$ is a quantum inverse semigroup. From now on, assume that $\Bbbk$ is a field of characteristic 0 (to maintain consistency with the definition of QISG given in \cite{qua}).

\begin{dfn}\cite[Definition 3.1]{qua}
     A quantum inverse semigroup (QISG) is a triple $(H, \Delta, \mathcal{S})$ in which: \begin{itemize}
    \item[(QISG1)] $H$ is a (not necessarily unital) $\Bbbk$-ring.
    \item[(QISG2)] $\Delta: H \longto H \otimes H$ is multiplicative and coassociative.
    \item[(QISG3)] $\mathcal{S}: H \longto H$ is a $\Bbbk$-linear map, called \emph{pseudo-antipode}, satisfying
    \subitem(i) $\mathcal{S}(h k)=\mathcal{S}(k) \mathcal{S}(h)$, for all $h, k \in H$.
    \subitem(ii) $I * \mathcal{S} * I=I$ and $\mathcal{S} * I * \mathcal{S}=\mathcal{S}$ in the convolution algebra $\operatorname{End}_{\Bbbk}(H)$.
    \item[(QISG4)] The sub $\Bbbk$-rings generated by the images of $I * \mathcal{S}$ and $\mathcal{S} * I$ mutually commute, that is, for every $h, k \in H$,
    \[
    h_{1} \mathcal{S}\left(h_{2}\right) \mathcal{S}\left(k_{1}\right) k_{2}=\mathcal{S}\left(k_{1}\right) k_{2} h_{1} \mathcal{S}\left(h_{2}\right)
    \]
    \end{itemize}
    
    A QISG is unital if $H$ is a unital $\Bbbk$-algebra and $\mathcal{S}\left(1_{H}\right)=1_{H}$. A QISG is co-unital if $H$ is a $\Bbbk$-coalgebra and $\varepsilon_{H} \circ \mathcal{S}=\varepsilon_{H}$.
\end{dfn}

For readers interested in learning more about this structure, we refer them to \cite{qua}, where properties and several examples are presented.

\begin{teo}\label{QISG}
    Let $H$ be a cocommutative weak Hopf algebra over $\Bbbk$. Then, the Hopf algebroid $\Hp$ has the structure of a unital QISG.
\end{teo}
\begin{proof}
    The proof follows similar steps as in \cite[Theorem 3.12]{qua} with the appropriate adaptations to the weak case. For the reader's convenience, at each step of the proof, we will highlight where the properties of the weak case appear. For this, consider the following linear map:
    \[
    \delta \colon H \longto \delta(H) \subseteq \Hp \otimes \Hp,
    \quad
    h \longmapsto\left[h_{1}\right] \otimes\left[h_{2}\right].
    \]
    
    We shall demonstrate that the map $\delta$ is a partial representation of $H$ in $\delta(H)$. Note that using \autoref{HRHLprep}, $[1_1]\otimes [1_2]$ is the unit in $\delta(H)$, hence it is clear that \eqref{defprep-1} holds. We will verify axioms \eqref{defprep-2} and \eqref{defprep-6}. The remaining axioms follow from similar arguments. 

    Through this proof, we use (*) to denote when the cocommutative property is used.
    
    For $h, k \in H$, we have that:
    \begin{align*}
    \delta(h) \delta\left(k_{1}\right) \delta\left(S\left(k_{2}\right)\right)
    &=
    \left[h_{1}\right]\left[k_{1}\right]\left[S\left(k_{4}\right)\right] \otimes\left[h_{2}\right]\left[k_{2}\right]\left[S\left(k_{3}\right)\right] \\
    \overset{(*)}&{=}
     \left[h_{1}\right]\left[k_{1}\right]\left[S\left(k_{2}\right)\right] \otimes\left[h_{2}\right]\left[k_{3}\right]\left[S\left(k_{4}\right)\right]\\
    \overset{\eqref{defHpar-2}}&{=}
    \left[h_{1} k_{1}\right]\left[S\left(k_{2}\right)\right] \otimes\left[h_{2} k_{3}\right]\left[S\left(k_{4}\right)\right] \\
    \overset{(*)}&{=}
    \left[h_{1} k_{1}\right]\left[S\left(k_{4}\right)\right] \otimes\left[h_{2} k_{2}\right]\left[S\left(k_{3}\right)\right] \\
    &=
    \delta\left(h k_{1}\right) \delta\left(S\left(k_{2}\right)\right),
    \end{align*}
    which proves \eqref{defprep-2}. Also,
    \begin{align*}
    \delta(h_1) \delta\left(S(h_{2})\right) \delta\left(h_{3}\right)
    &= 
    \left[h_{1}\right]\left[S(h_{4})\right]\left[h_{5}\right] \otimes\left[h_{2}\right]\left[S(h_{3})\right]\left[h_{6}\right] \\
    \overset{(*)}&{=}
    \left[h_{1}\right]\left[S(h_{2})\right]\left[h_{3}\right] \otimes\left[h_{4}\right]\left[S(h_{5})\right]\left[h_{6}\right] \\
    \overset{\eqref{defHpar-6}}&{=}
    \left[h_1\right] \otimes\left[h_{2}\right] \\
    &=
    \delta\left(h\right),
    \end{align*}
    
    Therefore, by \autoref{universaldeHpar}, there exists a unique algebra morphism $\Delta: \Hp \longto \delta(H) \subseteq \Hp \otimes \Hp$ given by
    \[
       \Delta\left(\left[h\right]\right)=\left[h_{1}\right] \otimes\left[h_{2}\right].
    \]
    In particular, $\Delta$ is also an algebra morphism from $\Hp$ to $\Hp \otimes \Hp$. Observe that $\Delta$ is induced by the comultiplication of $H$, so it is coassociative. Moreover, $\Delta$ is not necessarily unit-preserving in $\Hp\otimes\Hp$.
    
    In order to define the pseudo-antipode, consider the linear map
    \[
    \widetilde{S}: H \longto (\Hp)^{\mathrm{op}}, \quad h \longmapsto [S(h)]
    \]
    
    We shall prove that the map $\widetilde{S}$ is a partial representation. We will verify \eqref{defprep-1}, \eqref{defprep-2} and \eqref{defprep-6}. The axiom \eqref{defprep-1} is easy to check, as $\tilde{S}(1_H) = [S(1_H)] \stackrel{\autoref{p11}\eqref{p11-4}}{=} [1_H]$. Next, for all $h, k \in H$, 
    \[
    \begin{aligned}
    \tilde{S}(h) \cdot_{\text {op }} \tilde{S}\left(k_{1}\right) \cdot_{\text {op }} \tilde{S}\left(S\left(k_{2}\right)\right)
    &=
    [S(h)] \cdot_{\text {op }}\left[S\left(k_{1}\right)\right] \cdot_{\text {op }}\left[S\left(S\left(k_{2}\right)\right)\right] \\
    &=
    \left[S\left(S(k)_{1}\right)\right]\left[S(k)_{2}\right][S(h)] \\
    \overset{\eqref{defHpar-5}}&=
    \left[S\left(S(k)_{1}\right)\right]\left[S(k)_{2} S(h)\right] \\
    & =\left[S\left(S\left(k_{2}\right)\right)\right]\left[S\left(h k_{1}\right)\right] \\
    & =\left[S\left(h k_{1}\right)\right] \cdot_{\text {op }}\left[S\left(S\left(k_{2}\right)\right)\right] \\
    & =\widetilde{S}\left(h k_{1}\right) \cdot_{\text {op }} \tilde{S}\left(S\left(k_{2}\right)\right)
    \end{aligned}
    \]
    and so \eqref{defprep-2} is proved. Now we shall verify \eqref{defprep-6}. For every $h, k \in H$, 
    \[
    \begin{aligned}
    \tilde{S}(h_1) \cdot_{\text {op }} \tilde{S}\left(S(h_{2})\right) \cdot_{\text {op }} \tilde{S}\left(h_{3}\right)
    &=
    [S(h_1)] \cdot_{\text {op }}\left[S\left(S(h_{2})\right)\right] \cdot_{\text {op }}\left[S\left(h_{3}\right)\right] \\
    &=
    \left[S\left(h_{3}\right)\right]\left[S\left(S(h_{2})\right)\right][S(h_1)] \\
    &=
    \left[S\left(h\right)_1\right]\left[S\left(S(h)\right)_2\right][S(h)_3] \\
    \overset{\eqref{defHpar-6}}&{=}
    [S(h)] = \tilde{S}(h),
    \end{aligned}
    \]
    which shows \eqref{defprep-6}. Therefore, $\widetilde{S}$ is a partial representation of $H$ in $(\Hp)^{\mathrm{op}}$, inducing, by \autoref{universaldeHpar}, a morphism of algebras $S : \Hp \longto (\Hp)^{\mathrm{op}}$, or equivalently, an anti-morphism of algebras $S: \Hp \longto \Hp$ given by
    \[
    S\left(\left[h^{1}\right] \ldots\left[h^{n}\right]\right)=\left[S\left(h^{n}\right)\right] \ldots\left[S\left(h^{1}\right)\right]
    \]
    
    Next we shall demonstrate $I * S * I=I$ and $S * I * S=S$.

    Let us prove the identity $I * S * I\left(\left[h^{1}\right] \ldots\left[h^{n}\right]\right)=\left[h^{1}\right] \ldots\left[h^{n}\right]$ by induction on $n \geq 1$. For $n=1$, we have
    
    \[
    I * S * I([h])=\left[h_{1}\right]\left[S\left(h_{2}\right)\right]\left[h_{3}\right] \Ref{defHpar-6}[h]
    \]
    
    Assume the claim valid for $n$; then,
    
    \[
    \begin{aligned}
    &I * S * I\left(\left[h^{1}\right] \ldots\left[h^{n+1}\right]\right) = \\
    &=
    \left[h_{1}^{1}\right] \ldots\left[h_{1}^{n+1}\right]\left[S\left(h_{2}^{n+1}\right)\right] \ldots\left[S\left(h_{2}^{1}\right)\right]\left[h_{3}^{1}\right] \ldots\left[h_{3}^{n+1}\right] \\
    &=
    \left[h_{1}^{1}\right] \ldots\left[h_{1}^{n}\right] E_{h_{1}^{n+1}}\left[S\left(h_{2}^{n}\right)\right] \ldots\left[S\left(h_{2}^{1}\right)\right]\left[h_{3}^{1}\right] \ldots\left[h_{3}^{n}\right]\left[h_{3}^{n+1}\right] \\
    \overset{\ref{propee'}\ref{propee'-2}}&{=} E_{h_{1}^{1} \ldots h_{1}^{n} h_{1}^{n+1}}\left[h_{2}^{1}\right] \ldots\left[h_{2}^{n}\right]\left[S\left(h_{3}^{n}\right)\right] \ldots\left[S\left(h_{3}^{1}\right)\right]\left[h_{4}^{1}\right] \ldots\left[h_{4}^{n}\right]\left[h_{3}^{n+1}\right] \\
    \overset{I.H.}&{=}
    E_{h_{1}^{1} \ldots h_{1}^{n} h_{1}^{n+1}}\left[h_{2}^{1}\right] \ldots\left[h_{2}^{n}\right]\left[h_{2}^{n+1}\right] \\
    \overset{\ref{propee'}\ref{propee'-2}}&{=}
    \left[h^{1}\right] \ldots\left[h^{n}\right] E_{h_{1}^{n+1}}\left[h_{2}^{n+1}\right] \\
    &=
    \left[h^{1}\right] \ldots\left[h^{n}\right]\left[h_{1}^{n+1}\right]\left[S\left(h_{2}^{n+1}\right)\right]\left[h_{3}^{n+1}\right] \\
    \overset{\eqref{defHpar-6}}&=
    \left[h^{1}\right] \ldots\left[h^{n+1}\right] .
    \end{aligned}
    \]
    For the identity $S * I * S=S$, consider $\left[h^{1}\right] \ldots\left[h^{n}\right] \in \Hp$, and then
    \[
    \begin{aligned}
    &S * I * S\left(\left[h^{1}\right] \ldots\left[h^{n}\right]\right) \\
    & =\left[S\left(h_{1}^{n}\right)\right] \ldots\left[S\left(h_{1}^{1}\right)\right]\left[h_{2}^{1}\right] \ldots\left[h_{2}^{n}\right]\left[S\left(h_{3}^{n}\right)\right] \ldots\left[S\left(h_{3}^{1}\right)\right] \\
    & =\left[S\left(h_{3}^{n}\right)\right] \ldots\left[S\left(h_{3}^{1}\right)\right]\left[S\left(S\left(h_{2}^{1}\right)\right)\right] \ldots\left[S\left(S\left(h_{2}^{n}\right)\right)\right]\left[S\left(h_{1}^{n}\right)\right] \ldots\left[S\left(h_{1}^{1}\right)\right] \\
    & =\left[S\left(h^{n}\right)_{1}\right] \ldots\left[S\left(h^{1}\right)_{1}\right]\left[S\left(S\left(h^{1}\right)_{2}\right)\right] \ldots\left[S\left(S\left(h^{n}\right)_{2}\right)\right]\left[S\left(h^{n}\right)_{(3)}\right] \ldots\left[S\left(h^{1}\right)_{(3)}\right] \\
    & =\left[S\left(h^{n}\right)\right] \ldots\left[S\left(h^{1}\right)\right] \\
    & =S\left(\left[h^{1}\right] \ldots\left[h^{n}\right]\right) .
    \end{aligned}
    \]
    
    Finally, in order to verify axiom (QISG4), note that
    
    \[
    \begin{aligned}
    & I * S\left(\left[h^{1}\right] \ldots\left[h^{n}\right]\right)=\left[h_{1}^{1}\right] \ldots\left[h_{1}^{n}\right]\left[S\left(h_{2}^{n}\right)\right] \ldots\left[S\left(h_{2}^{1}\right)\right] = \\
    &=
    \left[h_{1}^{1}\right] \ldots\left[h_{1}^{n-1}\right] E_{h^{n}}\left[S\left(h_{2}^{n-1}\right)\right] \ldots\left[S\left(h_{2}^{1}\right)\right] \\
    &=
    E_{h_{1}^{1} \ldots h_{1}^{n-1} h^{n}}\left[h_{2}^{1}\right] \ldots\left[h_{2}^{n-1}\right]\left[S\left(h_{3}^{n-1}\right)\right] \ldots\left[S\left(h_{3}^{1}\right)\right] \\
    &=
    E_{h_{1}^{1} \ldots h_{1}^{n-1} h^{n}}\left[h_{2}^{1}\right] \ldots\left[h_{2}^{n-2}\right] E_{h_{2}^{n-1}}\left[S\left(h_{3}^{n-2}\right)\right] \ldots\left[S\left(h_{3}^{1}\right)\right] \\
    &=
    E_{h_{1}^{1} \ldots h_{1}^{n-1} h^{n}} E_{h_{2}^{1} \ldots h_{2}^{n-2} h_{2}^{n-1}}\left[h_{3}^{1}\right] \ldots\left[h_{3}^{n-2}\right]\left[S\left(h_{4}^{n-2}\right)\right] \ldots\left[S\left(h_{4}^{1}\right)\right]\\
    &= \ldots = E_{h_{1}^{1} \ldots h_{1}^{n-1} h^{n}} E_{h_{2}^{1} \ldots h_{2}^{n-2} h_{2}^{n-1}}\ldots E_{h_n^1}
    \end{aligned}
    \]
       
    while, on the other hand,
    
    \[
    \begin{aligned}
    S * I\left(\left[h^{1}\right] \ldots\left[h^{n}\right]\right) & =\left[S\left(h_{1}^{n}\right)\right] \ldots\left[S\left(h_{1}^{1}\right)\right]\left[h_{2}^{1}\right] \ldots\left[h_{2}^{n}\right] \\
    & =\left[S\left(h_{2}^{n}\right)\right] \ldots\left[S\left(h_{2}^{1}\right)\right]\left[S\left(S\left(h_{1}^{1}\right)\right)\right] \ldots\left[S\left(S\left(h_{1}^{n}\right)\right)\right] \\
    & =\left[S\left(h^{n}\right)_{1}\right] \ldots\left[S\left(h^{1}\right)_{1}\right]\left[S\left(S\left(h^{1}\right)_{2}\right)\right] \ldots\left[S\left(S\left(h^{n}\right)_{2}\right)\right]
    \end{aligned}
    \]
    
    \[
    \begin{aligned}
    & =E_{S\left(h^{n}\right)_{1} \ldots S\left(h^{2}\right)_{1}} S\left(h^{1}\right) E_{S\left(h^{n}\right)_{2}} \ldots S\left(h^{2}\right)_{2} \ldots E_{S\left(h^{n}\right)_{n}} \\
    & =E_{S\left(h_{n}^{n}\right) \ldots S\left(h_{n}^{2}\right) S\left(h^{1}\right)} E_{S\left(h_{n-1}^{n}\right) \ldots S\left(h_{n-1}^{2}\right)} \ldots E_{S\left(h_{1}^{n}\right)}
    \end{aligned}
    \]
    
    As both expressions can be written in terms of products of elements $E_{h}$, for $h \in H$, then they commute among themselves by \autoref{propee'} \ref{propee'-14}.  Therefore, for a cocommutative Hopf algebra $H$, the Hopf algebroid $\Hp$ is a QISG. Furthermore, it is clearly unital.
\end{proof}

\section*{Declarations}

\subsection*{Author's Contribution}
All authors wrote the main manuscript text and reviewed the manuscript.

\subsection*{Ethical Approval}
Not applicable.

\subsection*{Competing Interests}
 The authors have no competing interests as defined by Springer, or other interests that might be perceived to influence the results and/or discussion reported in this paper.

 \subsection*{Availability of Data and Materials}
 Data sharing not applicable to this article as no datasets were generated or analysed during the current study.

 \subsection*{Funding}
 There was no funding on the realization of this work.

\bibliographystyle{abbrv}

\end{document}